\documentclass[a4paper,11pt]{article}

\usepackage{geometry}
\usepackage{latexsym,amsmath,amssymb,amsfonts,epsfig,graphicx,cite,psfrag}
\usepackage{eepic,color,colordvi,amscd}
\usepackage{mathrsfs}
\usepackage{amsthm}
\usepackage{enumerate}
\usepackage{indentfirst}
\usepackage{tikz}
\usepackage{subfigure}
\usepackage[english]{babel}
\usepackage[center]{caption2}
\usepackage{multirow}
\usepackage[usenames,dvipsnames]{pstricks}
\usepackage{pst-grad} 
\usepackage{pst-plot} 
\usepackage[space]{grffile} 
\usepackage{etoolbox} 
\usepackage{float}
\usepackage{subfigure}
\usepackage{appendix}
\usepackage{hyperref}
\usepackage{booktabs}
\usepackage{array}
\usepackage{chngcntr}
\makeatletter 
\patchcmd\Gread@eps{\@inputcheck#1 }{\@inputcheck"#1"\relax}{}{}

\hypersetup{hidelinks,colorlinks=true,allcolors=black,pdfstartview=Fit,breaklinks=true}

\renewcommand{\thesubfigure}{(\roman{subfigure})}
\makeatletter \renewcommand{\@thesubfigure}{\thesubfigure \space}

\def\qedB{{\hfill\enspace\vrule height8pt depth0pt width8pt}}

\newtheorem{theo}{Theorem}[section]

\newtheorem{lem}[theo]{Lemma}
\newtheorem{cla}[theo]{Claim}
\newtheorem{cons}[theo]{Construction}

\newtheorem{prop}[theo]{Proposition}
\newtheorem{conj}[theo]{Conjecture}
\newtheorem{defi}[theo]{Definition}

\newcommand{\N}{{\mathbb{N}}}

\begin{document}

\title{Towards the Erd\H{o}s matching conjecture for 4-uniform hypergraphs: stability and applications}

\date{}

\author{
Peter Frankl\thanks{R\'{e}nyi Institute, Budapest, Hungary. Email: frankl.peter@renyi.hu.}~~~~~~~
Hongliang Lu\thanks{School of Mathematics and Statistics,
Xi'an Jiaotong University, Xi'an, Shaanxi, 710049, China. Research supported by National Key Research and Development Program of China 2023YFA1010203, National Natural Science Foundation of China under grant 12271425. Email: luhongliang215@sina.com.}~~~~~~~
Jie Ma\thanks{School of Mathematical Sciences, University of Science and Technology of China, Hefei, Anhui 230026, and Yau Mathematical Sciences Center, Tsinghua University, Beijing 100084, China. Research supported by National Key Research and Development Program of China  2023YFA1010201 and National Natural Science Foundation of China grant 12125106. Email: jiema@ustc.edu.cn.}~~~~~~~
Yuze Wu\thanks{School of Mathematical Sciences, University of Science and Technology of China, Hefei, Anhui, 230026, China. Research supported by Innovation Program for Quantum Science and Technology-National Science and Technology Major Project 2021ZD0302902. Email: lttch@mail.ustc.edu.cn.}}

\maketitle

\begin{abstract}
A famous conjecture of Erd\H{o}s asserts that for $k\ge 3$, the maximum number of edges in an $n$-vertex $k$-uniform hypergraph without $s+1$ pairwise disjoint edges is 
$\max\{\binom{n}{k}-\binom{n-s}{k},\binom{sk+k-1}{k}\}$. 
This problem has been central in extremal combinatorics, with substantial progress in the literature, including a complete solution for $k=3$ due to the first author. 
In this paper, we make progress towards the $4$-uniform case, proving the conjecture for $n\ge 5s$ and sufficiently large $n$, thereby taking a first step analogous to the $3$-uniform case. 
The main technical contribution is a stability result of independent interest. 
We further apply this stability to resolve two new instances of conjectures on the minimum $d$-degree threshold for matchings in $5$- and $6$-uniform hypergraphs, in a strengthened form.
\end{abstract}


\section{Introduction}

\noindent A \emph{$k$-uniform hypergraph} (or simply, a \emph{$k$-graph}) $H$ consists of a vertex set $V(H)$ and an edge set $E(H) \subseteq \binom{V(H)}{k}$, i.e., a collection of $k$-subsets of $V(H)$.
A \emph{matching} of size $s$ in $H$ is a collection of $s$ pairwise disjoint edges, and it is \emph{perfect} if its size equals $|V(H)|/k$.
The \emph{matching number} of a $k$-graph $H$, denoted by $\nu(H)$, is the maximum size of a matching in $H$.
Throughout this paper, let $n,k,s$ be positive integers with $n\geq k(s+1)$ and define $[n]:=\{1,2,\dots,n\}$. 
For each $i\in [k]$, define the $k$-graph $H_{i}=([n],E_{i})$, where the edge set $E_{i}$ is given by:
\begin{displaymath}
	E_{i}=\left\{e\subseteq\binom{[n]}{k}:\left|e\cap[i(s+1)-1]\right|\geq i\right\}.
\end{displaymath}
    It is straightforward to verify that $\nu(H_{i})=s$ for all $i\in [k]$.

A central question in extremal combinatorics is the following: what is the maximum number $e(H)$ of edges in an $n$-vertex $k$-graph $H$ under the constraint $\nu(H)\le s$? 
Erd\H{o}s \cite{E65} famously conjectured that this maximum is attained by one of two natural extremal constructions, namely $H_1$ or $H_k$. More precisely, he proposed the following.

\begin{conj}[Erd\H{o}s Matching Conjecture (\emph{EMC} for short), \cite{E65}]
	Let $H$ be an $n$-vertex $k$-graph with matching number at most $s$. Then
\begin{equation}\label{EMC}
	e(H)\le \max\left\{\binom{n}{k}-\binom{n-s}{k},\binom{sk+k-1}{k}\right\}.
\end{equation}
\end{conj}

In the original paper \cite{E65}, Erd\H{o}s proved \eqref{EMC} for all $n \ge n_{0}(k,s)$, for some sufficiently large $n_{0}(k,s)$. This result was enhanced by Bollob\'{a}s, Daykin, and Erd\H{o}s \cite{BDE76}, who demonstrated the EMC for $n \ge 2k^{3}s$. 
More than three decades later, Hao, Loh, and Sudakov \cite{HLS12} further advanced the result to $n \ge 3k^{2}s$.  
Subsequently, Frankl proved the EMC for $n \ge (2s+1)k - s$ in \cite{F13} and for $n \le (s+1)(k+\varepsilon)$ in \cite{F17A}, where $\varepsilon$ depends solely on $k$. 
More recently, Frankl and Kupavskii \cite{FK22} refined Frankl's arguments in \cite{F13} and established the EMC for sufficiently large $s$ and all $n \ge 5sk/3 - 2s/3$, 
which remains the best known general result to date.

Some special cases of the EMC are of independent interest. 
The case $s=1$ corresponds to the classical Erd\H{o}s--Ko--Rado theorem \cite{EKR61}, which is a cornerstone of extremal combinatorics and continues to motivate a substantial body of research. 
The EMC is trivial for $k=1$, and the graph case ($k=2$) was settled by Erd\H{o}s and Gallai \cite{EG59}. 
The first substantial progress for $3$-uniform hypergraphs was made by Frankl, R\"{o}dl, and Ruci\'{n}ski \cite{FRR12}, who proved the conjecture for $n \ge 4s$. 
Subsequently, \L uczak and Mieczkowska \cite{LM14} established the $3$-uniform case for all $n$ provided that $s \ge s_{0}$. 
Finally, Frankl \cite{F17B} gave a complete resolution of the conjecture for $k=3$.

In this paper, we investigate the EMC for $4$-uniform hypergraphs and some related extremal problems. 
Our main result, stated below, establishes the EMC for $k=4$ and all sufficiently large $n \ge 5s$, in a manner analogous to the result of Frankl, R\"{o}dl, and Ruci\'{n}ski \cite{FRR12}.

\begin{theo}\label{EMCK4}
Let $n, s$ be integers such that $n \ge 5s$ and $n \ge n_0$ for some absolute constant $n_0$. 
Let $H$ be an $n$-vertex $4$-graph with matching number at most $s$. Then $e(H)\leq\binom{n}{4}-\binom{n-s}{4}$. 
\end{theo}

The main technical result (from which the above theorem follows) we prove is a stability result, 
which applies to all integers $s$ in the range $cn\leq s\leq n/5$ and sufficiently large $n$,
where $c\in (0,1/5)$ is a constant. 
To state it precisely, we begin with the following two definitions.

\begin{defi}\label{stable}
	Let $n\geq k\geq 2$. A $k$-graph $G$ on $[n]$ is called \textbf{stable} if for any $\{a_{1},a_{2},\dots,a_{k}\}\in E(G)$ and any $\{b_{1},b_{2},\dots,b_{k}\}\in\binom{[n]}{k}$, the condition $b_{i}\leq a_{i}$ for all $1\leq i\leq k$ implies $\{b_{1},b_{2},\dots,b_{k}\}\in E(G)$.
\end{defi}

\begin{defi}\label{close}
	Let $k$ be a positive integer and let $G,H$ be two $k$-graphs on the same vertex set of size $n$. For any real number $\varepsilon>0$, we say that $G$ is \textbf{$\varepsilon$-close} to $H$ if $|E(H)\setminus E(G)|<\varepsilon n^{k}$.
\end{defi}
    Note that this definition is asymmetric.

\begin{theo}\label{StableK4}
	Fix any constant $c\in (0,1/5)$. Then for any sufficiently small $0<\varepsilon\ll c$, there exists $n_{0}\in \N$ such that the following holds for all integers $n,s$ with $n\ge n_{0}$ and $cn\leq s\leq n/5$. If $G$ is a stable $4$-graph on $[n]$ with $\nu(G)\leq s$ and $e(G)\geq \binom{n}{4}-\binom{n-s}{4}-\varepsilon n^{4}$, then $G$ is $400\varepsilon^{1/4}$-close to $H_{1}$.
\end{theo}

\noindent The proof of this result relies on a structural theorem for dense $4$-graphs with a given fractional matching number, which we refer to as a fractional stability result (see Theorem~\ref{FStableK4}). This provides a new approach to studying such extremal problems and may be of independent interest.

We also apply this stability result to other well-studied extremal problems concerning matchings. 
Let $H$ be a $k$-graph, and let $d$ be an integer with $1\leq d\leq k-1$. 
For any $S\in\binom{V(H)}{d}$, let $\textrm{deg}_{H}(S)=|\{e\in E(H): S\subset e\}|$ denote the \emph{degree} of $S$ in $H$. 
The \emph{minimum $d$-degree} of $H$ is then given by $\delta_{d}(H):=\min_{S\in\binom{V(H)}{d}}\textrm{deg}_{H}(S)$.
For integers $n,s$ with $0\leq s\leq n/k$, let $m_{d}^{s}(k,n)$ denote the minimum integer $m$ such that every $n$-vertex $k$-graph $H$ with $\delta_{d}(H)\geq m$ contains a matching of size $s$. 
The following conjecture was introduced by K\"{u}hn, Osthus, and Townsend in \cite{KOT14}.

\begin{conj}\label{KOT}
For any $\varepsilon>0$ and all integers $n,k,d,s$ with $1\leq d\leq k-1$ and $0\leq s\leq(1-\varepsilon)n/k$,
\begin{displaymath}
	m_{d}^{s}(k,n)=\left(1-\left(1-\frac{s}{n}\right)^{k-d}+o(1)\right)\binom{n-d}{k-d}.
\end{displaymath}
\end{conj}

\noindent For the interesting case of perfect matchings (i.e., $s=n/k$), the following conjecture was proposed earlier by H\`{a}n, Person, and Schacht in \cite{HPS09} and independently by K\"{u}hn and Osthus in \cite{KO09}.

\begin{conj}\label{HPSKO}
	Let $n,k,d$ be positive integers such that $1\leq d\leq k-1$ and $n/k\in\N$. Then
\begin{displaymath}
	m_{d}^{n/k}(k,n)=\left(\max\left\{1-\left(\frac{k-1}{k}\right)^{k-d}, \frac{1}{2}\right\}+o(1)\right)\binom{n-d}{k-d}.
\end{displaymath}
\end{conj}

\noindent The first term in Conjecture \ref{HPSKO} corresponds to the $k$-graph $H_{1}$, while the second term arises from a parity-based construction presented by K\"{u}hn and Osthus (see \cite{KO06}).

Alon et al. \cite{AFHRRS12} proved Conjecture \ref{HPSKO} for $k-4\leq d \leq k-1$ by reducing it to a probabilistic conjecture of Samuels, thereby settling the case $k \le 5$. 
For $d = k-1$, the exact value of $m_d^{n/k}(k,n)$ was determined for sufficiently large $n$ by R\"{o}dl, Ruci\'{n}ski, and Szemer\'{e}di \cite{RRS09}. 
Treglown and Zhao \cite{TZ13} subsequently extended this, resolving the conjecture whenever $k \in 4\mathbb{Z}$ and $d \ge k/2$, refining earlier asymptotic bounds from \cite{P08,RRS06,RRS09}.
For $d < k/2$, progress has been incremental.  
Concerning Conjecture~\ref{HPSKO}, the case $(k,d) = (3,1)$ was resolved independently by Khan \cite{K13} and by Kühn, Osthus, and Treglown \cite{KOT13}, while $(k,d) = (4,1)$ was settled by Khan \cite{K16}. 
Treglown and Zhao \cite{TZ16} further resolved $(k,d) \in \{(5,2), (7,3)\}$.
Han \cite{H16B} provided the best known upper bound on $m_d^{n/k}(k,n)$ for general $d$, establishing exact values for $0.42k \le d < k/2$ and for $(k,d) \in \{(12,5), (17,7)\}$.  
Lu and Yu \cite{LY22} subsequently extended this to $0.4k \le d < k/2$. 
Additional results appear in \cite{CGHW22,DH81,GLJ23,HPS09,HZ17,H16A,KOT14,LYY21,MR11}; Table~\ref{table} provides a summary of all known progress on $m_d^s(k,n)$ for sufficiently large $n$.

\begin{table}[htbp]
\belowrulesep=0pt \aboverulesep=0pt
\centering
\caption{Known progress on $m_{d}^{s}(k,n)$.}
\label{table}
\begin{tabular}{@{} 
    >{\raggedright\arraybackslash}p{2.7cm}|
    >{\raggedright\arraybackslash}p{6.0cm}|
    >{\raggedright\arraybackslash}p{4.5cm}|
    >{\centering\arraybackslash}p{1.8cm}
@{}}
\toprule
 &  \textbf{The Parameter} $s$ & \textbf{Parameters} $k,d$ & \textbf{Ref.}\\
\midrule[\heavyrulewidth]


\multirow{2}{3.4cm}{\textbf{Asymptotic Results}}
& $s=n/k$ & $k-4\leq d\leq k-1$ & \cite{AFHRRS12}\\
\cline{2-4}
& $s<\min\{1/2(k-d),(1-o(1))/k\}n$ & $1\leq d\leq k-2$ & \cite{KOT14}\\
\midrule[\heavyrulewidth]


\multirow{9}{3.4cm}{\textbf{Exact Results}}
& \multirow{6}{=}[-0.3em]{$s=n/k$}
& $k/2\leq d\leq k-1$ & \cite{RRS09,TZ13}\\
\cline{3-4}
& & $0.4k\leq d<k/2$ & \cite{H16B,LY22}\\
\cline{3-4}
& & $(k,d)=(3,1)$ & \cite{K13,KOT13}\\
\cline{3-4}
& & $(k,d)=(4,1)$ & \cite{K16}\\
\cline{3-4}
& & $(k,d)\in\{(5,2),(7,3)\}$ & \cite{TZ16}\\
\cline{2-4}
& $s\leq \lfloor n/k\rfloor -1$ & $k/2<d\leq k-1$ & \cite{LYY21,CGHW22}\\
\cline{2-4}
& $s\leq(k/2(k-1)-o(1))n/k$ & $k\geq 3$, $d=1$ & \cite{GLJ23}\\
\cline{2-4}
& $s\leq n/k$ & $(k,d)\in\{(5,1),(6,2)\}$ & $\bigstar$\\

\bottomrule
\addlinespace
\multicolumn{4}{@{}p{\dimexpr\textwidth-2\tabcolsep}@{}}{%
\footnotesize Prior results are cited accordingly, and new contributions established in this paper are marked with $\bigstar$.} \\
\end{tabular}
\end{table}

Using Theorem~\ref{StableK4}, we determine the exact value of $m_d^s(k,n)$ for $(k,d) \in \{(5,1),(6,2)\}$ for all sufficiently large $n$ and all integers $0 \le s \le n/k$ in the following statement, which confirms Conjectures~\ref{KOT} and \ref{HPSKO} in the corresponding cases in their precise forms.\footnote{The asymptotic form of Conjecture~\ref{HPSKO} for the case $(5,1)$ was previously established in \cite{AFHRRS12}.}

\begin{theo}\label{VDPC}
Let $(k,d) \in \{(5,1),(6,2)\}$. For all sufficiently large $n$ and all $s$ with $0 \le s \le n/k$,  
\[
m_d^s(k,n) = \binom{n-d}{k-d} - \binom{n-d-s+1}{k-d} + 1.
\]
\end{theo}

The remainder of the paper is organized as follows. In Section~2, we prove a fractional version of Theorem~\ref{StableK4}, using several auxiliary lemmas whose proofs are deferred to Appendix~\ref{sec:Appendix}. In Section~3, we apply this fractional result to establish the stability theorem (Theorem~\ref{StableK4}), which is the key ingredient in the proofs of our main results. In Section~4, we derive Theorem~\ref{EMCK4} as a direct consequence of the stability theorem. Finally, Section~5 contains the proof of Theorem~\ref{VDPC}.

\section{A fractional stability}
\noindent In this section, we establish a fractional version of Theorem~\ref{StableK4}, which plays a key role in its proof.
We begin with some necessary definitions.
A \emph{fractional matching} in a $k$-graph $H=(V,E)$ is a function $\phi : E\to [0,1]$ such that for each $v\in V$, $\sum_{v\in e}\phi(e)\leq 1$. 
The quantity $\sum_{e\in E}\phi(e)$ is called the size of $\phi$. 
The \emph{fractional matching number} of $H$, denoted by $\nu^{*}(H)$, is the maximum possible size of a fractional matching in $H$. If $\nu^{*}(H)=|V|/k$, or equivalently, for all $v\in V$ we have $\sum_{v\in e}\phi(e)=1$, then we call $\phi$ a \emph{perfect} fractional matching.

We now state the main result of this section. 
We emphasize that the positive constant $\delta$ in the statement is essential for deriving Theorem~\ref{StableK4} from Theorem~\ref{FStableK4} (see Section~\ref{sec:mainpf}).

\begin{theo}\label{FStableK4}
	Let $c\in(0,1/5)$ be a constant and set $\delta=10^{-10}$. Then for any sufficiently small  $0<\varepsilon\ll c$, there exists $n_{0}\in \N$ such that the following holds for all integers $n,s$ with $n\geq n_{0}$ and $cn\leq s\leq (n-3)/(5-\delta)$. 
    If $G$ is a stable $4$-graph on $[n]$ with $\nu^{*}(G)\leq s$ and $e(G)\geq \binom{n}{4}-\binom{n-s}{4}-\varepsilon n^{4}$, then $G$ is $100\varepsilon^{1/4}$-close to the extremal $4$-graph $H_{1}$.
\end{theo}	

From a technical standpoint, the proof of Theorem~\ref{FStableK4} relies crucially on the resolution of the case $k=3$ of the EMC (see the proof of Claim~\ref{claim:sum-upp}). 

\subsection{Preliminary results}

\noindent Before proving Theorem~\ref{FStableK4}, we first establish several preliminary results.
  
The determination of $\nu^{*}(H)$ is a classical linear programming problem, and we therefore also consider its dual.
A \emph{fractional vertex cover} in a $k$-graph $H=(V,E)$ is a function $\omega \colon V \to [0,1]$ such that
$\sum_{v \in e} \omega(v) \ge 1$ for every $e\in E.$
The quantity $\sum_{v \in V} \omega(v)$ is called the \emph{size} of $\omega$.
The \emph{fractional covering number} of $H$, denoted by $\tau^{*}(H)$, is the minimum possible size of a fractional vertex cover in $H$.
Then, for every $k$-graph $H$, the Duality Theorem implies that 
\[
\nu(H) \le \nu^{*}(H) = \tau^{*}(H) \le \tau(H),
\]
where $\tau(H)$ denotes the \emph{vertex cover number} of $H$, that is, the minimum size of a subset of $V$ that intersects every edge of $H$.
By the complementary slackness conditions of linear programming, we obtain the following proposition.

\begin{prop}\label{CS}
	Let $n$ and $k$ be positive integers. Consider a $k$-graph $H$ on $[n]$, and let $s=\nu^{*}(H)=\tau^{*}(H)$. Suppose $\omega: [n] \to [0,1]$ is a minimum fractional vertex cover of $H$, and define its support as $R=\{i\in[n] : \omega(i)>0\}$. Then $|R|\le ks$.
\end{prop}

\begin{proof}
	Let $\phi : E(H)\to [0,1]$ be a maximum fractional matching of $H$. By complementary slackness, for each $i\in[n]$, we have $\omega(i)>0$ if and only if $\sum_{i\in e\in E}\phi(e)=1$. Consequently, 
\begin{displaymath}
    ks=k\sum_{e\in E(H)}\phi(e)=\sum_{i\in[n]}\sum_{i\in e}\phi(e)\geq \sum_{i\in R}\sum_{i\in e}\phi(e)=|R|,
\end{displaymath}
    which completes our proof.
\end{proof}

The proof of Theorem~\ref{FStableK4} relies on the following five lemmas.  
The first is a classical result.
The remaining auxiliary lemmas (Lemmas~\ref{convex}, \ref{Prebound}, \ref{Maxvalue}, and \ref{Calculate}) are proved in Appendix~\ref{sec:Appendix}.

\begin{lem}[Frankl \cite{F13}]\label{Lemma 1}
	For $\mathcal{F}\subseteq \binom{[n]}{k}$, where $n,k$ are positive integers, let $\partial\mathcal{F}=\{F\in\binom{[n]}{k-1}: F\subseteq F' \textrm{ for some } F'\in\mathcal{F}\}$. Then $|\mathcal{F}|\leq \nu(\mathcal{F})|\partial\mathcal{F}|$.
\end{lem}

\begin{lem}\label{convex}
	Let $m,k,s$ be integers with $k\geq 3$, $s\geq1$ and $m\geq ks+k-2$. Suppose that
\begin{displaymath}
	f(x)=\max\left\{\binom{m}{k-1}-\binom{m-\frac{(1-a)s}{1-x}}{k-1},\binom{\frac{(k-1)(1-a)s}{1-x}+k-2}{k-1}\right\},
\end{displaymath}
with parameter $a\in[0,1)$ and variable $x\in[0,1)$. Then $f(x)$ is convex on $[0,(a+1)/2]$.
\end{lem}

\begin{lem}\label{Prebound}
	Let $G$ be a $3$-graph on $[m]$, $\delta=10^{-10}$ and $\mu\leq 1$ be a positive real number. If $G$ has a fractional vertex cover $\omega: [m]\to[0,1]$ with the following properties :
	\begin{itemize}
		\item $3/5\geq\omega(1)\geq\omega(2)\geq\dots\geq\omega(m)$.
		\item $\sum_{i=1}^{m}\omega(i)\leq\mu s$, $s\leq (m-2)/(4-\delta)$.
		\item $\omega(3s)=\omega(3s+1)=\dots=\omega(m)=0$.
	\end{itemize}
	Then $e(G)\leq\max\{\binom{3s-1}{3}-\binom{3s-1-\mu s}{3},\binom{3\mu s+2}{3}\}+\binom{2\mu s}{2}\binom{m-3s+1}{1}$.
	Moreover, if we further assume that $\omega(1)<1/2$, then $e(G)\leq\max\{\binom{3s-1}{3}-\binom{3s-1-\mu s}{3},\binom{3\mu s+2}{3}\}$.
\end{lem}

\begin{lem}\label{Maxvalue}
	Let $m$ be a positive integer and $\varepsilon<10^{-20}$ be a positive real number. Let $a,b$ be real numbers satisfying $1/4\leq b\leq a<1-5\varepsilon^{1/4}$. Set $\delta=10^{-10}$, $\mu=(1-a)/(1-b)$ and $\beta=1-\delta+3a-(4-\delta)b$. Suppose that
\begin{displaymath}
	f(s)=\left\{\begin{array}{ll}
		\binom{m}{3}-\binom{m-\mu s}{3}, & \textrm{ for } b\geq 3/8\\[0.3em]
		\max\{\binom{3s-1}{3}-\binom{3s-1-\mu s}{3},\binom{3\mu s+2}{3}\}+\binom{2\mu s}{2}\binom{m-3s+1}{1}, &  \textrm{ for } 1/3\leq b< 3/8\\[0.3em]
		\max\{\binom{3s-1}{3}-\binom{3s-1-\mu s}{3},\binom{3\mu s+2}{3}\}, & \textrm{ for } 1/4\leq b< 1/3 \textrm{ },
	\end{array}
	\right. 
\end{displaymath}
	and that
\begin{displaymath}
	g(s)=(1-a)s f(s)-\frac{(1-\delta)(1-a)s}{\beta}\binom{m-\mu s}{3}.
\end{displaymath}
	Then for sufficiently large $m$, all $0\leq p\leq q\leq\frac{m-2}{4-\delta}$ and all $s\in[p,q]$, we have $g(s)\leq \max\{g(p),g(q)\}$.
\end{lem}

\begin{lem}\label{Calculate}
	Let $\varepsilon<10^{-20}$ be a positive real number and $a,b$ be real numbers satisfying $1/4\leq b\leq a<1-5\varepsilon^{1/4}$. Set $\delta=10^{-10}$, $\mu=(1-a)/(1-b)$ and $\beta=1-\delta+3a-(4-\delta)b$. Let 
\begin{displaymath}
g(a,b):=f(a,b)-\frac{1-\delta}{\beta}\left(1-\frac{\mu}{4-\delta}\right)^{3}, \mbox{ where }
\end{displaymath}
\begin{displaymath}
f(a,b)=\left\{
\begin{array}{ll}
    1-(1-\frac{\mu}{4-\delta})^{3}, & b\geq 3/8\\[0.3em]
    (4-\delta)^{-3}\cdot \left(\max\{27-(3-\mu)^{3},27\mu^{3}\}+3(1-\delta)(4-\delta)\mu^{2}\right), & 1/3\leq b <3/8\\[0.3em]
    (4-\delta)^{-3}\cdot \max\{27-(3-\mu)^{3},27\mu^{3}\}, & 1/4\leq b<1/3 \textrm{ }.
\end{array}\right.
\end{displaymath}
Then $g(a,b)<-52\varepsilon^{3/4}$. Furthermore, for $b<3/8$, we have $g(a,b)<-\frac{1}{8100}$.
\end{lem}

\medskip

In the rest of this section, we present the proof of Theorem~\ref{FStableK4}, assuming the above lemmas hold.

\subsection{Proof of Theorem~\ref{FStableK4}}
\noindent We devote this subsection to the proof of Theorem \ref{FStableK4}.
	Assume that $\varepsilon<\min\{(\frac{c}{1000})^{4}, 10^{-20}\}$. 
    Let $n$ be sufficiently large and $cn\leq s\leq (n-3)/(5-\delta)$. 
    Let $G$ be a stable $4$-graph on $[n]$ with $\nu^{*}(G)\leq s$ and $e(G)\geq \binom{n}{4}-\binom{n-s}{4}-\varepsilon n^{4}$.
    Let $\omega : [n]\to [0,1]$ be a minimum fractional vertex cover of $G$. 
    Then we have $\sum_{i=1}^{n}\omega(i)= \tau^{*}(G)=\nu^{*}(G)\leq s$.
    By the minimality of $\omega$ and since $G$ is stable, we have $\omega(1)\geq \omega(2)\geq\dots\geq\omega(n)$.
    Under this condition, we may assume that for any $e\in\binom{[n]}{4}$ with $\sum_{i\in e}\omega(i)\geq 1$, we always have $e\in E(G)$.  
    By Proposition \ref{CS}, we also have $\omega(4s+1)=\dots=\omega(n)=0.$

	For any subset $A\subseteq [s+1]$, let $\mathcal{F}(A)=\{e\setminus A: e\in E(G),e\cap[s+1]=A\}$ and $\mathcal{H}(A)=\{e\setminus A : e\in E(H_{1}), e\cap[s+1]=A\}$. Then for any subset $A\subseteq [s+1]$ with $2\leq |A|\leq 4$, we have $|\mathcal{F}(A)|\leq \binom{n-s-1}{4-|A|}=|\mathcal{H}(A)|$. Since $e(G)\geq \binom{n}{4}-\binom{n-s}{4}-\varepsilon n^{4}=e(H_{1})-\varepsilon n^{4}$, 
\begin{equation}\label{lowerbound}
	|\mathcal{F}(\emptyset)|+\sum_{i=1}^{s+1}|\mathcal{F}(\{i\})|\geq s\binom{n-s-1}{3}-\varepsilon n^{4}.
\end{equation}
	
	Let $a,b\in[0,1]$ be real numbers satisfying $\sum_{i=1}^{s}\omega(i)=as$ and $\omega(s+1)=b$. 
    If $b<1/4$, then for any $f\in\binom{[n]\setminus[s]}{4}$, we have $\sum_{i\in f}\omega(i)\leq 4b<1$, implying $f\not\in E(G)$. 
    Therefore we must have $E(G)\subseteq E(H_{1})$. 
    Since $e(G)\geq \binom{n}{4}-\binom{n-s}{4}-\varepsilon n^{4}$, 
    clearly $G$ is $\varepsilon$-close to $H_{1}$ and thus $G$ is $100\varepsilon^{1/4}$-close to $ H_{1}$. 
    If $a=1$, then we have $b=0<\frac{1}{4}$. Similarly, $G$ is $100\varepsilon^{1/4}$-close to $ H_{1}$.

	From now on, we assume that $1/4\leq b\leq a<1$ and that $G$ is not $100\varepsilon^{1/4}$-close to $H_{1}$.
    In the following proof, we aim to find a contradiction to \eqref{lowerbound}, by showing that
\begin{equation}\label{equ:uppbound}
	|\mathcal{F}(\emptyset)|+\sum_{i=1}^{s+1}|\mathcal{F}(\{i\})|< s\binom{n-s-1}{3}-\varepsilon n^{4}.
\end{equation}

	First, consider the case when $1-5\varepsilon^{1/4}\leq a<1$.
	Note that for all $1\leq i\leq s$, if $\omega(i)=1$, then $|\mathcal{F}(\{i\})|=\binom{n-s-1}{3}$. Since $s\leq(n-3)/(5-\delta)$, there exists at least $100\varepsilon^{1/4} n^{4}/\binom{n-s-1}{3}>600\varepsilon^{1/4} n$ different vertices in $[s]$ whose weight is less than $1$.
	
	Note that for all $1\leq i\leq s$, if $\omega(i)<1$, then $|\mathcal{F}(\{i\})|\leq\binom{n-s-1}{3}-\binom{n-4s}{3}$. Also note that $\nu(\mathcal{F}(\emptyset))\leq \nu^{*}(\mathcal{F}(\emptyset))\leq (1-a)s$. By Lemma \ref{Lemma 1}, we have 
\begin{displaymath}
    |\mathcal{F}(\emptyset)|\leq \nu\left(\mathcal{F}(\emptyset)\right)\cdot|\partial\mathcal{F}(\emptyset)|\leq (1-a)s \cdot\binom{n-s-1}{3}\leq \frac{5\varepsilon^{1/4}}{6(5-\delta)}n^{4}<\frac{5\varepsilon^{1/4}}{24} n^{4}.
\end{displaymath}
    Since $|\mathcal{F}(\{s+1\})|=O(n^{3})$, we have $|\mathcal{F}(\{s+1\})|\leq \varepsilon n^{4}$ for sufficiently large $n$.
	Then, for sufficiently large $n$, we have the desired inequality:
\begin{displaymath}
	\begin{array}{lll}
	|\mathcal{F}(\emptyset)|+\sum_{i=1}^{s+1}|\mathcal{F}(\{i\})| & \leq & s\binom{n-s-1}{3}-600\varepsilon^{1/4}n\binom{n-4s}{3}+\frac{5}{24}\varepsilon^{1/4}n^{4}+\varepsilon n^{4}\\[0.3em]
	& \leq & s\binom{n-s-1}{3} -\frac{100(1-\delta)^{3}}{(5-\delta)^{3}}\varepsilon^{1/4} n^{4}+\frac{5}{24}\varepsilon^{1/4}n^{4}+\varepsilon n^{4}\\[0.3em]
	& < & s\binom{n-s-1}{3} -\frac{4}{5}\varepsilon^{1/4} n^{4}+\frac{5}{24}\varepsilon^{1/4}n^{4}+\varepsilon n^{4}\\[0.3em]
	& < & s\binom{n-s-1}{3} -\frac{1}{2}\varepsilon^{1/4} n^{4} < s\binom{n-s-1}{3} -\varepsilon n^{4},
	\end{array}
\end{displaymath}
	a contradiction to \eqref{lowerbound}.

	It suffices to consider the case when $\frac{1}{4}\leq b\leq a<1-5\varepsilon^{1/4}$.
	By Lemma \ref{Lemma 1} and since $G$ is stable, we have $|\mathcal{F}(\emptyset)|\leq \nu(G\setminus[s+1])\cdot |\mathcal{F}(\{s+1\})|$.
	Note that $\nu(G\setminus[s+1])\leq \nu^{*}(G\setminus[s+1])\leq s-as$. So
	\begin{equation}\label{equ:upp1}
    |\mathcal{F}(\emptyset)|+\sum_{i=1}^{s+1}|\mathcal{F}(\{i\})|\leq \sum_{i=1}^{s+1}|\mathcal{F}(\{i\})|+(1-a)s\cdot |\mathcal{F}(\{s+1\})|.
	\end{equation}

\begin{cla}\label{claim:sum-upp}
	Let $m:=n-s-1$, $\mu:=(1-a)/(1-b)$, and $\beta:=1-\delta+3a-(4-\delta)b$. Then 
\begin{equation}\label{equ:sum-upp}
	|\mathcal{F}(\emptyset)|+\sum_{i=1}^{s+1}|\mathcal{F}(\{i\})|\leq (s+1)\binom{m}{3}+(1-a)s\cdot |\mathcal{F}(\{s+1\})|-\frac{(1-\delta)(1-a)s}{\beta}\binom{m-\mu s}{3}.
\end{equation}
\end{cla}

\begin{proof}[Proof of Claim \ref{claim:sum-upp}.]
Throughout this proof, we define the increasing function $f(x)$ on $[0,1)$, by letting
$$f(x):=\max\left\{\binom{m}{3}-\binom{m-\frac{(1-a)s}{1-x}}{3},\binom{\frac{3(1-a)s}{1-x}+2}{3}\right\}.$$
Since the EMC holds for 3-graphs (see Frankl \cite{F17B}) and by the fact that $\nu(\mathcal{F}(\{i\}))\leq\nu^{*}(\mathcal{F}(\{i\}))\leq \frac{(1-a)s}{1-\omega(i)}$, for every $i\in [s+1]$, we have
\begin{align*}
|\mathcal{F}(\{i\})|\leq \max\left\{\binom{m}{3}-\binom{m-\nu(\mathcal{F}(\{i\}))}{3},\binom{3\nu(\mathcal{F}(\{i\}))+2}{3}\right\}\leq f(\omega(i)),
\end{align*}
and hence $|\mathcal{F}(\{i\})|\leq \min \{f(\omega(i)),\binom{m}{3}\}$.
Note that $\omega(i)\geq \omega(s+1)=b$ for every $i\in [s]$.
Let $A:=\{i\in [s]: \omega(i)\in[b,\frac{1-\delta+3a}{4-\delta}]\}$. 
Since $s\leq (n-3)/(5-\delta)$ and $m=n-s-1$, we can derive $(4-\delta)s+2\leq m$. 
Then since $f(x)$ is increasing, for any $i\in A$, we have
$$f(\omega(i))\leq f\left(\frac{1-\delta+3a}{4-\delta}\right)=\max\left\{\binom{m}{3}-\binom{m-(4-\delta)s/3}{3}, \binom{(4-\delta)s+2}{3}\right\}\leq \binom{m}{3}.$$
Since $s\le \frac{m-2}{4-\delta}$ and $\mu=\frac{1-a}{1-b}\le 1$, a direct calculation shows that
$\binom{m}{3}-\binom{m-\mu s}{3} > \binom{3\mu s+2}{3},$ implying $$f(b)=\binom{m}{3}-\binom{m-\mu s}{3}.$$ 
By Lemma~\ref{convex}, $f(x)$ is convex on $[0,(a+1)/2]$ and thus it is convex on $[b,\frac{1-\delta+3a}{4-\delta}]$. 
Therefore, for any $b\leq x<y\leq\frac{1-\delta+3a}{4-\delta}$, by letting $\rho:=\min\{x-b,\frac{1-\delta+3a}{4-\delta}-y\}$, we have $f(x)+f(y)\leq f(x-\rho)+f(y+\rho)$. 
Repeatedly applying this operation, one can show that there exist an integer $\lambda\geq 0$ and some $z\in\left[b,\frac{1-\delta+3a}{4-\delta}\right]$ such that $ \left(|A|-\lambda-1\right)\frac{1-\delta+3a}{4-\delta}+z+\lambda b=\sum_{i\in A}\omega(i)$ and
\begin{equation*}
    \sum_{i\in A}f(\omega(i))\leq (|A|-\lambda-1)\cdot f\left(\frac{1-\delta+3a}{4-\delta}\right)+f(z)+\lambda f(b).
\end{equation*}
Note that $\min\left\{f(\omega(i)),\binom{m}{3}\right\}=f(\omega(i))$ for $i\in A$ and $\min\left\{f(\omega(i)),\binom{m}{3}\right\}\leq \binom{m}{3}$ for $i\not\in A$.
The above bounds,  together with the fact $f(z)\leq f\left(\frac{1-\delta+3a}{4-\delta}\right)\leq \binom{m}{3}$, imply that
\begin{align*}\label{step1}
\sum_{i=1}^{s+1}\min\left\{f(\omega(i)),\binom{m}{3}\right\}
&\leq (s-|A|)\binom{m}{3}+\sum_{i\in A}f(\omega(i))+f(b)\\
&\leq (s-\lambda)\binom{m}{3}+(\lambda+1)\left(\binom{m}{3}-\binom{m-\mu s}{3}\right).
\end{align*} 
As for $\lambda$, we have $(s-\lambda-1)\frac{1-\delta+3a}{4-\delta}+(\lambda+1)b\leq \sum_{i=1}^{s}\omega(i)=as$, which is equivalent to that 
$$\lambda+1\geq\frac{(1-\delta)(1-a)s}{\beta}, \mbox{ where } \beta=1-\delta+3a-(4-\delta)b.$$
Finally, the last two inequalities show that
\begin{align*}
	\sum_{i=1}^{s+1} |\mathcal{F}(\{i\})|\leq \sum_{i=1}^{s+1}\min\left\{f(\omega(i)),\binom{m}{3}\right\}
    \leq (s+1)\binom{m}{3}-\frac{(1-\delta)(1-a)s}{\beta}\binom{m-\mu s}{3}.
\end{align*}
This, together with \eqref{equ:upp1}, provides the desired inequality, proving the claim. 
\end{proof}

\begin{cla}\label{claim}
	For sufficiently large $n$, we have
\begin{displaymath}
	\binom{m}{3}+(1-a)s\cdot |\mathcal{F}(\{s+1\})|-\frac{(1-\delta)(1-a)s}{\beta}\binom{m-\mu s}{3}<-\varepsilon n^{4}.
\end{displaymath}
\end{cla}

\begin{proof}[Proof of Claim \ref{claim}]
	Note that $\binom{m}{3}=O(n^{3})$. So, for sufficiently large $n$, we have $\binom{m}{3}\leq \varepsilon n^{4}$. Also note that $m=n-s-1\geq \frac{(4-\delta)n-(2-\delta)}{5-\delta}>\frac{10}{13}n$. Thus, it suffices to show that for sufficiently large $m$,
\begin{displaymath}
	(1-a)s\cdot |\mathcal{F}(\{s+1\})|-\frac{(1-\delta)(1-a)s}{\beta}\binom{m-\mu s}{3}<-\frac{28561}{5000}\varepsilon m^{4}.
\end{displaymath}

	Since $\omega$ is a fractional vertex cover of $G$, there exists a fractional vertex cover $\omega^{*}: [n]\setminus[s+1]\to[0,1]$ of $\mathcal{F}(\{s+1\})$ with $\omega^{*}(i)=\omega(i)/(1-b)$ for all $s+2\leq i\leq n$. Note that $\sum_{i=s+2}^{n}\omega^{*}(i)\leq (s-as)/(1-b)=\mu s$ and $b/(1-b)\geq \omega^{*}(s+2)\geq\dots\geq\omega^{*}(n)$. So for $b<1/3$, we have $\omega^{*}(s+2)<1/2$. By Lemma \ref{Prebound}, 
	$$|\mathcal{F}(\{s+1\})|\leq \max\left\{\binom{3s-1}{3}-\binom{3s-1-\mu s}{3},\binom{3\mu s+2}{3}\right\}.$$
	Similarly, for $1/3\leq b <3/8$, we obtain $\omega^{*}(s+2)< 3/5$. Then also by Lemma \ref{Prebound}, we have
	$$|\mathcal{F}(\{s+1\})|\leq \max\left\{\binom{3s-1}{3}-\binom{3s-1-\mu s}{3},\binom{3\mu s+2}{3}\right\}+\binom{2\mu s}{2}\binom{m-3s+1}{1}.$$
	For $b\geq 3/8$, since $\mu s\leq s\leq (m-2)/(4-\delta)$, we have the following bound.
	$$|\mathcal{F}(\{s+1\})|\leq \binom{m}{3}-\binom{m-\mu s}{3}.$$
Let
\begin{displaymath}
	h_{0}(s)=\left\{\begin{array}{ll}
		\binom{m}{3}-\binom{m-\mu s}{3}, & b\geq 3/8\\[0.3em]
		\max\{\binom{3s-1}{3}-\binom{3s-1-\mu s}{3},\binom{3\mu s+2}{3}\}+\binom{2\mu s}{2}\binom{m-3s+1}{1}, & 1/3\leq b< 3/8\\[0.3em]
		\max\{\binom{3s-1}{3}-\binom{3s-1-\mu s}{3},\binom{3\mu s+2}{3}\}, & 1/4\leq b< 1/3 \textrm{ },
	\end{array}
	\right. 
\end{displaymath}
	and
\begin{displaymath}
	h(s)=(1-a)s h_{0}(s)-\frac{(1-\delta)(1-a)s}{\beta}\binom{m-\mu s}{3}.
\end{displaymath}
It is enough to show that for sufficiently large $m$, we have
\begin{displaymath}
	h(s)<-\frac{28561}{5000}\varepsilon m^{4}.
\end{displaymath}
Note that $(m-2)/(4-\delta)\geq s\geq cn\geq cm >1000\varepsilon^{1/4} m$. 
    By Lemma \ref{Maxvalue}, it suffices to prove that
\begin{displaymath}
	\max\left\{h(1000\varepsilon^{1/4}m),h\left(\frac{m-2}{4-\delta}\right)\right\}<-\frac{28561}{5000}\varepsilon m^{4}.
\end{displaymath}
	Let $p(m):=h(1000\varepsilon^{1/4}m)$ and $q(m):=h(\frac{m-2}{4-\delta})$ be two polynomials of degree $4$ with variable $m$.
    Let $C_{p},C_{q}$ be the coefficients of $m^{4}$ in $p(m)$ and $q(m)$, respectively. 
    It then suffices to prove that
\begin{displaymath}
	\max\{C_{p},C_{q}\}<-\frac{28561}{5000}\varepsilon .
\end{displaymath}

	Let $\eta=1000\varepsilon^{1/4}<10^{-2}$. By the definition of $h(s)$, we have
\begin{displaymath}
	C_{p}=\left\{\begin{array}{ll}
		\frac{\eta(1-a)}{6}\left\{[1-(1-\eta\mu)^{3}]-\frac{(1-\delta)(1-\eta\mu)^{3}}{\beta}\right\}, & b\geq 3/8\\[0.25em]
		\frac{\eta(1-a)}{6}\left\{\max\{27-(3-\mu)^{3},27\mu^{3}\}\eta^{3}+12\mu^{2}(1-3\eta)\eta^{2}-\frac{(1-\delta)(1-\eta\mu)^{3}}{\beta}\right\}, & 1/3\leq b< 3/8\\[0.25em]
		\frac{\eta(1-a)}{6}\left\{\max\{27-(3-\mu)^{3},27\mu^{3}\}\eta^{3}-\frac{(1-\delta)(1-\eta\mu)^{3}}{\beta}\right\}, & 1/4\leq b< 1/3 \textrm{ }.
	\end{array}
	\right. 
\end{displaymath}
    Note that $5\varepsilon^{1/4}<1-a\leq\mu\leq 1$, $\beta=1-\delta+3a-(4-\delta)b=(4-\delta)(1-b)-3(1-a)\in(5(1-\delta)\varepsilon^{1/4},4)$, $\delta=10^{-10}$, $\varepsilon<10^{-20}$ and $\eta=1000\varepsilon^{1/4}<10^{-2}$. We have
\begin{displaymath}
	C_{p}<\frac{\eta(1-a)}{6}\cdot \left(-\frac{1}{5}\right)<-\frac{500}{3}\varepsilon^{1/2}<-\frac{28561}{5000}\varepsilon.
\end{displaymath}
	By the definition of $h(s)$, we also have
\begin{displaymath}
	C_{q}=\left\{\begin{array}{ll}
		\frac{1-a}{6(4-\delta)}\left\{1-(1-\frac{\mu}{4-\delta})^{3}-\frac{1-\delta}{\beta}(1-\frac{\mu}{4-\delta})^{3}\right\}, & b\geq 3/8\\[0.3em]
		\frac{1-a}{6(4-\delta)}\left\{\frac{\max\{27-(3-\mu)^{3},27\mu^{3}\}+3(1-\delta)(4-\delta)\mu^{2}}{(4-\delta)^{3}}-\frac{1-\delta}{\beta}(1-\frac{\mu}{4-\delta})^{3}\right\}, & 1/3\leq b< 3/8\\[0.3em]
		\frac{1-a}{6(4-\delta)}\left\{\frac{\max\{27-(3-\mu)^{3},27\mu^{3}\}}{(4-\delta)^{3}}-\frac{1-\delta}{\beta}(1-\frac{\mu}{4-\delta})^{3}\right\}, & 1/4\leq b< 1/3 \textrm{ }.
	\end{array}
	\right. 
\end{displaymath}
	Then by Lemma \ref{Calculate}, we have
\begin{displaymath}
	C_{q}<\frac{(1-a)}{6(4-\delta)}\cdot (-52\varepsilon^{3/4})<-\frac{65}{6}\varepsilon<-\frac{28561}{5000}\varepsilon.
\end{displaymath}
Therefore, we have proved that $\max\{C_{p},C_{q}\}<-\frac{28561}{5000}\varepsilon,$
which completes the proof of Claim \ref{claim}.
\end{proof}

Now, Claims \ref{claim:sum-upp} and \ref{claim} together imply that
\begin{displaymath}
	|\mathcal{F}(\emptyset)|+\sum_{i=1}^{s+1}|\mathcal{F}(\{i\})|<s\binom{n-s-1}{3}-\varepsilon n^{4},
\end{displaymath}
	which contradicts \eqref{lowerbound}.
This final contradiction completes the proof of Theorem~\ref{FStableK4}.
\qedB

\section{Proof of Theorem~\ref{StableK4} (Stability)}\label{sec:mainpf}

\noindent This section is devoted to the proof of Theorem \ref{StableK4}.
    Throughout this section, we set $\varepsilon<\min\{(c/1000)^{4},10^{-20}\}$ and let $G$ be a stable $4$-graph on $[n]$ with $\nu(G)\leq s$ and $e(G)\geq \binom{n}{4}-\binom{n-s}{4}-\varepsilon n^{4}$. 
Suppose for a contradiction that $G$ is not $400\varepsilon^{1/4}$-close to $H_{1}$. 
We will complete the proof by finding a matching of size at least $s+1$ in $G$. 
To proceed, we first define the following $4$-graph $H$.

\begin{cons}\label{Cons1}
	Let $\eta<\varepsilon$ be a positive real number. Let $t=\lfloor(n-4s)/3-\eta n\rfloor$ and $H$ be the stable $4$-graph on $[n+t]$ such that $E(H)=\{e\in\binom{[n+t]}{4}: e\cap [t]\neq\emptyset\}\bigcup \{\{a_{1}+t,a_{2}+t,a_{3}+t,a_{4}+t\}: \{a_{1},a_{2},a_{3},a_{4}\}\in E(G)\}$.
\end{cons}

	By Construction \ref{Cons1}, we know that $e(H)\geq \binom{n+t}{4}-\binom{n-s}{4}-\varepsilon n^{4}$. Furthermore, we also have $\nu(H)=\nu(G)+t\leq (n+t)/4-3\eta n/4$. Thus, it suffices to find a matching in $H$ that covers all but at most $\sigma n$ vertices for some $\sigma<3\eta$. We will use the following classical lemma in \cite{FR85}.

\begin{lem}[Frankl and R\"{o}dl \cite{FR85}]\label{NPM}
	For every integer $k\geq 2$ and any real $\sigma>0$, there exists $\tau=\tau(k,\sigma)$ and $d_{0}=d_{0}(k,\sigma)$ such that, for every $n\geq D\geq d_{0}$ the following holds: Every $k$-graph $H$ on $n$ vertices with $(1-\tau)D<d_{H}(v)<(1+\tau)D$ for all $v\in V(H)$ and $\Delta_{2}(H):=\max_{S\subseteq\binom{V(H)}{2}}d_{H}(S)<\tau D$, contains a matching covering all but at most $\sigma n$ vertices.
\end{lem}

	To apply this powerful lemma, we will use probabilistic method to find a spanning subgraph $H''$ of $H$, which satisfies the conditions of Lemma \ref{NPM}. This idea was first proposed by Alon et.al. in \cite{AFHRRS12} and has been widely used in this area since then.

\begin{defi}\label{EX}
	For any integer $k\geq 2$, let $H(U,W)$ be the $k$-graph on $U\cup W$ such that $E(H(U,W))=\{e\in\binom{U\cup W}{k} : e\cap U\neq\emptyset\}$. In particular, $H([s],[n]\setminus[s])=H_{1}$.
\end{defi}

\begin{lem}[Chernoff's bound]\label{random}
	Suppose $X_{1},\dots,X_{n}$ are independent random variables taking values in $\{0,1\}$. Let $X=\sum_{i=1}^{n}X_{i}$ and $\mu=\mathbb{E}[X]$. Then for any $0<\delta\leq 1$
	$$\mathbb{P}[X\geq(1+\delta)\mu]\leq e^{-\delta^{2}\mu/3} \textrm{ and }\mathbb{P}[X\leq(1-\delta)\mu]\leq e^{-\delta^{2}\mu/2}.$$
	In particular, when $X\sim \textrm{Bi}(n,p)$ and $\lambda<\frac{3}{2}np$, then
	$$\mathbb{P}[|X-np|\geq \lambda]\leq e^{-\Omega(\lambda^{2}/np)}.$$
\end{lem}

\begin{lem}\label{random1}
	Let $H$ be the $4$-graph defined in Construction \ref{Cons1}. Let $\alpha=s/n$. Let $R$ be chosen from $V(H)$ by taking each vertex uniformly at random with probability $n^{-0.9}$, and then arbitrarily deleting less than $4$ vertices so that $|R|\in 4\mathbb{Z}$. Take $n^{1.1}$ independent copies of $R$ and denote them by $R^{i}$, $1\leq i\leq n^{1.1}$. Let $T^{i}=[t]\cap R^{i}$, $V^{i}=([t+s]\setminus[t])\cap R^{i}$ and $W^{i}=R^{i}\setminus T^{i}$. For each $A\subseteq V(H)$ with $|A|\leq 4$, let $Y_{A}=|\{i : A\subseteq R^{i}\}|$.
	Then with probability $1-o(1)$, all of the following statements hold:
\begin{enumerate}[(1)]
	\item For every $v\in V(H)$,$Y_{\{ v\}}=(1+o(1))n^{0.2}$.
	\item $Y_{\{ u,v\}}\leq 2$ for every pair $\{u,v\}\subseteq V(H)$.
	\item $Y_{e}\leq 1$ for every edge $e\in E(H)$.
	\item For all $1\leq i\leq n^{1.1}$, we have $|R^{i}|=(1+o(1))\frac{4-4\alpha-3\eta}{3}n^{0.1}$, $|T^{i}|=(1+o(1))\frac{1-4\alpha-3\eta}{3}n^{0.1}$, $|V^{i}|=(\alpha+o(1))n^{0.1}$ and $|W^{i}|=(1+o(1))n^{0.1}$.
	\item For all $1\leq i\leq n^{1.1}$, we have $e(H[R^{i}])\geq \binom{|R^{i}|}{4}-\binom{|R^{i}|-|T^{i}|-|V^{i}|}{4}-2\varepsilon |W^{i}|^{4}$, and $H[W^{i}]$ is not $200\varepsilon^{1/4}$-close to the graph $H(V^{i},W^{i}\setminus V^{i})$.
\end{enumerate}
\end{lem}

\begin{proof}[Proof of Lemma \ref{random1}]
	Note that the removal of less than $4$ vertices from each $R^{i}$ does not affect \textsl{(1)-(5)}. Also note that $|Y_{A}| \sim Bi(n^{1.1},n^{-0.9|A|})$ for $A\subseteq V(H)$. Thus, $\mathbb{E}[Y_{\{ v\}}]=n^{0.2}$ for $v\in V(H)$, and it follows from Lemma \ref{random} that
	$\mathbb{P}[|Y_{\{ v\}}-n^{0.2}|\geq n^{0.15}]\leq e^{-\Omega(n^{0.1})}.$
	Hence \textsl{(1)} holds with probability at least $1-e^{-\Omega(n^{0.1})}$.
	
	To prove \textsl{(2)} and \textsl{(3)}, let
	$Z_{2}=|\{\{u,v\}\in\binom{V(H)}{2}: Y_{\{u,v\}}\geq 3\}|,$
	and for $\ell\geq 3$, let
	$$Z_{\ell}=\left|\{A\in\binom{V(H)}{\ell}: Y_{A}\geq 2\}\right|.$$
	Then $\mathbb{E}[Z_{2}]<(n+t)^{2}(n^{1.1})^{3}(n^{-0.9})^{6}=(\frac{4-4\alpha-3\eta}{3})^{2}n^{-0.1}<2n^{-0.1}$ and $\mathbb{E}[Z_{\ell}]<(n+t)^{\ell}(n^{1.1})^{2}(n^{-0.9})^{2\ell}=(\frac{4-4\alpha-3\eta}{3})^{\ell}n^{2.2-0.8\ell}<n^{-0.1}$ (for $\ell\geq 3$ and $n$ is sufficiently large). By Markov's inequality,
	$$\mathbb{P}[Z_{2}=0]>1-2n^{-0.1}\textrm{ and }\mathbb{P}[Z_{\ell}=0]>1-n^{-0.1} \textrm{ for } \ell \geq 3.$$
	Thus \textsl{(2)} and \textsl{(3)} hold with probability at least $1-2n^{-0.1}$ and $1-n^{-0.1}$, respectively.

	By Lemma \ref{random} (with $\lambda=n^{0.095}$), we have
	$$\mathbb{P}\left[\left||R^{i}|-\frac{4-4\alpha-3\eta}{3}n^{0.1}\right|>n^{0.095}\right]\leq e^{-\Omega(n^{0.09})}$$
	for each $1\leq i\leq n^{1.1}$. Thus by the union bound, $|R^{i}|=(1+o(1))\frac{4-4\alpha-3\eta}{3}n^{0.1}$ holds for all $1\leq i\leq n^{1.1}$ with probability $1-n^{1.1}e^{-\Omega(n^{0.09})}$. Similarly, \textsl{(4)} holds with probability $1-4n^{1.1}e^{-\Omega(n^{0.09})}$.

	Next, we prove \textsl{(5)}. Let $\kappa=\frac{1-4\alpha-3\eta}{3}$. Conditioning on $\left||R^{i}|-(1+\kappa)n^{0.1}\right|<n^{0.095}$, $\left||T^{i}|-\kappa n^{0.1}\right|<n^{0.095}$, $\left||V^{i}|-\alpha n^{0.1}\right|<n^{0.095}$ and $\left||W^{i}|-n^{0.1}\right|<n^{0.095}$ for all $1\leq i\leq n^{1.1}$, and using the assumption that $n$ is sufficiently large, we have
\begin{displaymath}
	(400\varepsilon^{1/4} n^{4})(n^{-0.9})^{4}=400\varepsilon^{1/4} n^{0.4}\geq 300\varepsilon^{1/4}|W^{i}|^{4},
\end{displaymath}
\begin{displaymath}
	\left(\binom{(1+\kappa)n}{4}-\binom{(1+\kappa)n-s-t}{4}-\varepsilon n^{4}\right)(n^{-0.9})^{4}\geq\binom{|R^{i}|}{4}-\binom{|R^{i}|-|T^{i}|-|V^{i}|}{4}-\frac{4}{3}\varepsilon |W^{i}|^{4}.
\end{displaymath}
Note that $H$ is stable, so both $H[R^{i}]$ and $H[W^{i}]$ are stable, and for each fixed $R^{i}$, we have
\begin{displaymath}
	\begin{array}{lll}
	\mathbb{E}\left[\left|E(H(V^{i},W^{i}\setminus V^{i}))\setminus E(H[W^{i}]))\right|\right] & = & (1-o(1))\left|E(H([s],[n]\setminus[s]))\setminus E(G)\right|(n^{-0.9})^{4}\\[0.25em]
	& \geq &(1-o(1))(400\varepsilon^{1/4} n^{4})(n^{-0.9})^{4}\\[0.25em]
	& \geq &(1-o(1))300\varepsilon^{1/4}|W^{i}|^{4}\geq 250\varepsilon^{1/4}|W^{i}|^{4},
	\end{array}
\end{displaymath}
	and
\begin{displaymath}
	\begin{array}{lll}
 	\mathbb{E}\left[e(H[R^{i}])\right] & = & (1-o(1))e(H)(n^{-0.9})^{4}\\[0.3em]
	& \geq &(1-o(1))\left(\binom{(1+\kappa)n}{4}-\binom{(1+\kappa)n-s-t}{4}-\varepsilon n^{4}\right)(n^{-0.9})^{4}\\[0.3em]
	& \geq &(1-o(1))\left(\binom{|R^{i}|}{4}-\binom{|R^{i}|-|T^{i}|-|V^{i}|}{4}-\frac{4}{3}\varepsilon |W^{i}|^{4}\right)\\[0.3em]
	& \geq &\binom{|R^{i}|}{4}-\binom{|R^{i}|-|T^{i}|-|V^{i}|}{4}-\frac{5}{3}\varepsilon |W^{i}|^{4}.
	\end{array}
\end{displaymath}
	In particular,
	$\mathbb{E}[e(H[R^{i}])]=\Omega(n^{0.4})$ and $\mathbb{E}\left[\left|E(H(V^{i},W^{i}\setminus V^{i}))\setminus E(H[W^{i}])\right|\right]=\Omega(n^{0.4}).$
	Next we apply Janson's inequality to bound the deviation of $e(H[R^{i}])$ and $|E(H(V^{i},W^{i}\setminus V^{i}))\setminus E(H[W^{i}])|$.
	Write $e(H[R^{i}])=X=\sum_{e\in E(H)}X_{e}$, where $X_{e}=1$ if $e\subseteq R^{i}$ and $X_{e}=0$ otherwise. Similarly, write $|E(H(V^{i},W^{i}\setminus V^{i}))\setminus E(H[W^{i}])|=Y=\sum_{f=\{a_{1},a_{2},a_{3},a_{4}\}\in E(H_{1})\setminus E(G)}Y_{f}$, where $Y_{f}=1$ if $\{a_{1}+t,a_{2}+t,a_{3}+t,a_{4}+t\}\in E(H(V^{i},W^{i}\setminus V^{i}))\setminus E(H[W^{i}])$ and $Y_{f}=0$ otherwise.
	Then for $Z\in\{X,Y\}$,
	$$\Delta_{Z}=\sum_{e\cap f\neq\emptyset}\mathbb{P}[Z_{e}=Z_{f}=1]\leq(n^{-0.9})^{8-\ell}\binom{n+t}{4}\binom{4}{\ell}\binom{n+t-4}{4-\ell}=O(n^{0.7}).$$
	By Janson's inequality, for any $\gamma>0$ and $Z\in\{X,Y\}$, we have
	$$\mathbb{P}\left[Z\leq (1-\gamma)\mathbb{E}[Z]\right]\leq e^{-\gamma^{2}\mathbb{E}[Z]/(2+\Delta_{Z}/\mathbb{E}[Z])}=e^{-\Omega(n^{0.1})}.$$
	Taking $\gamma$ sufficiently small, the union bound shows that with probability at least $1-n^{1.1}e^{-\Omega(n^{0.1})}$,
	$$e(H[R^{i}])\geq \binom{|R^{i}|}{4}-\binom{|R^{i}|-|T^{i}|-|V^{i}|}{4}-2\varepsilon |W^{i}|^{4}$$
	holds for all $1\leq i\leq n^{1.1}$. Moreover, with probability at least $1-n^{1.1}e^{-\Omega(n^{0.1})}$, the statement that $H[W^{i}]$ is not $200\varepsilon^{1/4}$-close to the graph $H(V^{i},W^{i}\setminus V^{i})$, also holds for all $1\leq i\leq n^{1.1}$.
	Therefore \textsl{(5)} holds with probability at least
	$$(1-n^{1.1}e^{-\Omega(n^{0.1})})^{2}>1-2n^{1.1}e^{-\Omega(n^{0.1})}.$$

	Finally, it follows from the union bound that with probability at least
	$$1-e^{-\Omega(n^{0.1})}-3n^{-0.1}-4n^{1.1}e^{-\Omega(n^{0.09})}-2n^{1.1}e^{-\Omega(n^{0.1})}=1-o(1),$$
	all \textsl{(1)-(5)} hold. This completes the proof of this lemma. 
\end{proof}

\begin{lem}\label{PFM}
	Let $0<\eta<\varepsilon<10^{-20}$ be two real numbers. Let $n,s,t$ be positive integers such that $s=(1+o(1))\alpha n$ and $t=(1+o(1))\frac{1-4\alpha-3\eta}{3}n$ for some $\alpha\in(0,1/5]$. Suppose that $H$ is a stable $4$-graph on $[n+t]$ which satisfies the following statements.
\begin{enumerate}[(1)]
	\item $e(H)\geq \binom{n+t}{4}-\binom{n-s}{4}-2\varepsilon n^{4}$.
	\item For all $i \in[t]$, $i$ is a full degree vertex in $H$.
	\item $H\setminus[t]$ is not $200\varepsilon^{1/4}$-close to the graph $H([s+t]\setminus[t],[n]\setminus [s+t])$.
\end{enumerate}
	Then for sufficiently large $n$, $H$ contains a perfect fractional matching.
\end{lem}

\begin{proof}[Proof of Lemma \ref{PFM}]
	Let $H'=H\setminus[t]$.	 Then $H'$ is a stable $n$-vertex $4$-graph. We have $e(H')\geq\binom{n}{4}-\binom{n-s}{4}-2\varepsilon n^{4}$ and know that $H'$ is not $200\varepsilon^{1/4}$-close to the graph $H([s+t]\setminus[t],[n]\setminus [s+t])$.
	Let $s^{*}=(n+t)/4-t=(1+o(1))(\alpha+3\eta/4)n\geq s$. Then we have $cn<s^{*}<\frac{n-3}{5-10^{-10}}$ for some positive real constant $c$ and $e(H')\geq\binom{n}{4}-\binom{n-s^{*}}{4}-4\varepsilon n^{4}$. We also know that $H'$ is not $200\varepsilon^{1/4}$-close to the graph $H([s^{*}+t]\setminus[t],[n]\setminus [s^{*}+t])$. If $\nu^{*}(H')\leq s^{*}$, then by Theorem \ref{FStableK4}, we know that $H'$ is $100\sqrt{2}\varepsilon^{1/4}$-close to the graph $H([s^{*}+t]\setminus[t],[n]\setminus [s^{*}+t])$, a contradiction. So we have $\nu^{*}(H')\geq s^{*}$. Now we are going to construct a perfect fractional matching of $H$.

	Let $\phi: E(H')\to [0,1]$ be a fractional matching of $H'$ with size $s^{*}$, such that for each $i\in[n+t]\setminus[t]$, the value $\sum_{i\in e} \phi(e)$ is maximized subject to the condition that for every $j\in [t+i-1]\setminus[t]$, $\sum_{j\in e} \phi(e)$ is also maximized. For each $i\in[n+t]\setminus[t]$, let $\phi(i):=\sum_{i\in e\in E(H')}\phi(e)$.
 
\begin{cla}\label{dec}
	$\phi(t+1)\geq \phi(t+2)\geq\dots\geq\phi (t+n)$.
\end{cla}

\begin{proof}[Proof of Claim \ref{dec}]
	If there exists $i,j\in[n+t]\setminus[t]$ such that $i<j$ and that $\phi(i)<\phi(j)$. Then there exists $e\in E(H')$ which contains $j$, such that $\phi(e)>\phi(e')$, where $e'=(e\setminus\{j\})\cup\{i\}\in E(H')$. Let $\rho=\min\{1-\phi(i),\phi(e)-\phi(e')\}>0$. Then we replace $\phi(e)$ and $\phi(e')$ with $\phi(e)-\rho$ and $\phi(e')+\rho$ respectively. Note that the size of $\phi$ does not change and for all $p\in([n+t]\setminus[t])\setminus \{i,j\}$, $\phi(p)$ does not change as well. However, $\phi(i)$ strictly increases, which contradicts to our choice of $\phi$.
\end{proof}

\begin{cla}\label{BV}
	Let $A=\{i\in[n+t]\setminus [t]: 0<\phi(i)<1\}$, then $|A|\leq 4$.
\end{cla}

\begin{proof}[Proof of Claim \ref{BV}]
	By Claim \ref{dec}, we may assume that $A=\{q+1,\dots,q+\ell\}$. If $|A|\geq 5$, then there exists $e\in E(H')$ with $\phi(e)>0$ such that $\{q+j,q+j+1,\dots,q+\ell\}\subseteq e$ and $q+j-1\not\in e$ for some integer $\ell\geq j\geq 2$. 
Let $e'=(e\setminus\{q+j\})\cup\{q+j-1\}$. Since $H'$ is stable, we have $e'\in E(H')$. 
Let $\rho=\min\{1-\phi(q+j-1),\,\phi(e)\}>0.$
We then replace $\phi(e)$ and $\phi(e')$ by $\phi(e)-\rho$ and $\phi(e')+\rho$, respectively. Note that the size of $\phi$ remains unchanged, and for every
$p\in([n+t]\setminus[t])\setminus\{q+j-1,q+j\},$
the value of $\phi(p)$ is unchanged as well. 
However, the new value of $\phi(q+j-1)$ strictly increases, contradicting to the choice of $\phi$.
\end{proof}
    
	Assume that $A=\{q+1,q+2,\dots,q+\ell\}$. Then we have $\phi(i)=0$ for all $i>q+\ell$ and $\phi(i)=1$ for all $t<i\leq q$. Furthermore, we have $q=t+4s^{*}$ when $A=\emptyset$ and $t+4s^{*}-|A|<q<t+4s^{*}$ when $A\neq\emptyset$, where $4s^{*}=n-3t$ is an integer. For each $1\leq j\leq \ell$, let $a_{j}=\phi(q+j)\in(0,1)$. 
    
    Now we extend $\phi$ to $E(H)$, such that for any $i\in[n+t]$, $\phi(i):=\sum_{i\in e\in E(H)}\phi(e)=1$. For any $i\in[t]$, $i$ is a full degree vertex in $H$. So $\forall e\in\binom{[n+t]}{4}$, $i\in e$ implies $e\in E(H)$.
    For each $2\leq i\leq t$, let $\phi(\{i,t+4s^{*}+3i-2,t+4s^{*}+3i-1,t+4s^{*}+3i\})=1$. If $A=\emptyset$, then let $\phi(\{1,t+4s^{*}+1,t+4s^{*}+2,t+4s^{*}+3\})=1$. If $A\neq\emptyset$, then let $B=[t+4s^{*}+3]\setminus[q+\ell]$, $p=t+4s^{*}-q$ and $E_{0}=\{e\in E(H): 1\in e, B\subseteq e, |e\cap A|=|A|-p\}$. Observe that $1\leq p<|A|$, $|A|-p+|B|=3$ and $|E_{0}|=\binom{|A|}{|A|-p}\geq |A|=\ell$. Therefore, the following system of equations always admits a solution.
\begin{equation}\label{systemofeq}
    \sum_{q+j\in e\in E_{0}}\phi(e)=1-a_{j}, 1\leq j\leq l.
\end{equation}
    Let $\{\phi(e):e\in E_{0}\}$ be a solution of \eqref{systemofeq}. Then for each $e\in E_{0}$, we have $0\leq \phi(e)\leq 1$. Moreover, for each $i\in B$, $\sum_{i\in e\in E_{0}}\phi(e)=\sum_{1\in e\in E_{0}}\phi(e)=\sum_{e\in E_{0}}\phi(e)=1$. 
    
    Finally, for any remaining $e\in E(H)$, let $\phi(e)=0$. Then, $\phi$ is a fractional matching of $H$, and for each $i\in[n+t]$, $\phi(i)=\sum_{i\in e\in E(H)}\phi(e)=1$. Therefore, $\phi$ is a perfect fractional matching of $H$.
\end{proof}	

\begin{lem}[Alon et.al., \cite{AFHRRS12}]\label{random2}
	Let $H$ be a $k$-graph on $\Theta(n)$ vertices. Assume that $R^{i}\subseteq V(H)$, $1\leq i\leq n^{1.1}$, satisfy conditions $(1)$-$(3)$ in Lemma \ref{random1} and that each $H[R^{i}]$ has a perfect fractional matching $\omega^{i}$. Then there exists a spanning subgraph $H''$ of $H$ such that $d_{H''}(v)=(1+o(1))n^{0.2}$ for each $v\in V(H)$, and $\Delta_{2}(H'')\leq n^{0.1}$.
\end{lem}

	Now, we are ready to show Theorem \ref{StableK4}.

\medskip

\noindent{\it Proof of Theorem \ref{StableK4}.}
	Assume that $n$ is sufficiently large and that $G$ is a stable $4$-graph on $[n]$ with $\nu(G)\leq s$ and $e(G)\geq \binom{n}{4}-\binom{n-s}{4}-\varepsilon n^{4}$.
	If $G$ is not $400\varepsilon^{1/4}$-close to the graph $H_{1}$.
	Then we get a $4$-graph $H$ which satisfies the conditions in Construction \ref{Cons1}.

	By Lemma \ref{random1}, we get $R^{i}$ for all $1\leq i\leq n^{1.1}$, which satisfy $(1)$-$(5)$ in Lemma \ref{random1}.
	Therefore, by Lemma \ref{PFM}, for $1\leq i\leq n^{1.1}$, each $H[R^{i}]$ contains a perfect fractional matching.
	Thus, by Lemma \ref{random2}, we get a spanning subgraph $H''$ of $H$, such that $d_{H''}(v)=(1+o(1))n^{0.2}$ for each $v\in V(H)$, and $\Delta_{2}(H'')\leq n^{0.1}$.
	So, using Lemma \ref{NPM}, we have a matching covering all but at most $2\eta n$ vertices of $H$.
	Finally, by the definition of $H$, we have a matching covering all but at most $n-4s-\eta n$ vertices of $G$, implying that $\nu(G)>s$, a contradiction.
	Therefore, $G$ is $400\varepsilon^{1/4}$-close to $H_{1}$. \qedB

\section{Proof of Theorem \ref{EMCK4}}
\noindent We present the proof of Theorem \ref{EMCK4} in this section.

    Since the case $s \leq n/8$ is already covered by Frankl's result \cite{F13}, it suffices to prove Theorem \ref{EMCK4} for $n/8 \leq s \leq n/5$. Set $c = 1/8$ and $\varepsilon=10^{-40}$. Let $n_{0}$ be the integer provided by Theorem \ref{StableK4} for these fixed $c$ and $\varepsilon$. For integers $n>n_{0}$, let $G$ be a $4$-graph on $[n]$ with $e(G)>\binom{n}{4}-\binom{n-s}{4}$. Our aim is to find a matching of size $s+1$ in $G$. Without loss of generality, we may assume that $G$ is stable, by the following useful lemma of Frankl \cite{F87}.

\begin{lem}[Frankl, \cite{F87}]\label{stableF}
	Let $n,k$ be positive integers, and let $H$ be a $k$-graph on $[n]$. Then there exists a stable $k$-graph $H^{*}$ on $[n]$, such that $|E(H^{*})|=|E(H)|$ and that $\nu(H^{*})\leq \nu(H)$. 
\end{lem} 

    Assume that $\nu(G)\leq s$, then by Theorem \ref{StableK4}, $G$ is $400\varepsilon^{1/4}$-close to $H_{1}$. 
    
    Set $\eta=400\varepsilon^{1/4}=4\times10^{-10}$. For each $i\in[n]$, let $N_{G}(i)$ and $N_{H_{1}}(i)$ denote the collection of edges in $G$ and $H_{1}$ that contains $i$, respectively. Let $A=\{i\in[n]:|N_{H_{1}}(i)\setminus N_{G}(i) |\geq \sqrt{\eta} n^{3}\}$.
    Then we have $|A|\leq 4\sqrt{\eta} n$, otherwise
\begin{displaymath}
    |E(H_{1})\setminus E(G)|\geq\frac{1}{4} \sum_{i\in A} |N_{H_{1}}(i)\setminus N_{G}(i)|>\frac{1}{4}(4\sqrt{\eta}n)\sqrt{\eta}n^{3}=\eta n^{4},
\end{displaymath}
    a contradiction. Let $V_{1}=[s]\setminus A$, $t=|V_{1}|$ and $V_{2}=[n]\setminus V_{1}$.
\begin{cla}\label{M1}
    $G[V_{2}]$ contains a matching of size $s-t+1$.
\end{cla}
\begin{proof}[Proof of Claim \ref{M1}]
    The number of edges in $G$ that intersects $V_{1}$ is at most $\binom{n}{4}-\binom{n-t}{4}$. So we have $e(G[V_{2}])>(\binom{n}{4}-\binom{n-s}{4})-(\binom{n}{4}-\binom{n-t}{4})=\binom{n-t}{4}-\binom{n-t-(s-t)}{4}$.
    Note that $t\leq s\leq n/5$, $s-t\leq |A|\leq 4\sqrt{\eta} n$. Since $\sqrt{\eta}=2\times10^{-5}$, by Frankl's result \cite{F13}, $G[V_{2}]$ contains a matching of size $s-t+1$.
\end{proof}

    Let $M_{1}=\{f_{1},f_{2},\dots,f_{s-t+1}\}$ be a matching of size $s-t+1$ in $G[V_{2}]$ and let $U=\bigcup_{j=1}^{s-t+1}f_{j}$. Next we find disjoint edges $e_{1},e_{2},\dots,e_{t}$ in $G$ such that $|e_{i}\cap V_{1}|=1$ and $e_{i}\cap U=\emptyset$ for all $i\in[t]$. 
    Suppose for some $i\in[t-1]$, we have already found such edges $e_{1},e_{2},\dots,e_{i}$. Let $W_{i}=\bigcup_{j=1}^{i}e_{j}$, and $v$ be a vertex in $V_{1}\setminus W_{i}$, the number of edges in $G$ that contains $v$ and intersects $U\cup V_{1}\cup W_{i}$ is at most
    $\binom{n-1}{3}-\binom{n-4(s-t+1)-t-3i}{3}$. By the definition of $V_{1}$, we have $|N_{G}(v)|>\binom{n-1}{3}-\sqrt{\eta}n^{3}$. Since $n-4(s-t+1)-t-3i\geq n-4s-1>n/10$, we have $\binom{n-4(s-t+1)-t-3i}{3}>10^{-4}n^{3}>\sqrt{\eta}n^{3}$, implying $|N_{G}(v)|>\binom{n-1}{3}-\binom{n-4(s-t+1)-t-3i}{3}$. Thus, there exists an edge $e_{i+1}$ in $G$ such that $e_{i+1}\cap V_{1}=\{v\}$ and $e_{i+1}\cap(U\cup W_{i})=\emptyset$.   
    Finally, $e_{1},e_{2},\dots,e_{t}$ and $f_{1},f_{2},\dots,f_{s-t+1}$ form a matching of size $s+1$ in $G$, which completes the proof of Theorem~\ref{EMCK4}. \qedB

\bigskip

We would like to remark that, through a very coarse calculation, it suffices to choose $n_0=200/\varepsilon$ in Theorem~\ref{FStableK4}.
Using this bound, it is enough to take $n_0=10^{10^7}$ in the statement of Theorem~\ref{EMCK4}.

\section{Proof of Theorem~\ref{VDPC}}
\noindent Throughout this section, let \(n,k,d,s\) be positive integers with \((k,d)\in\{(5,1),(6,2)\}\) and \(0 \le s \le n/k - 1\). Let \(G\) be an \(n\)-vertex \(k\)-graph with
$\delta_d(G) \ge \binom{n-d}{k-d} - \binom{n-d-s}{k-d} + 1.$
Our goal is to find a matching in \(G\) of size at least \(s+1\), following an approach similar to those in \cite{LYY21} and \cite{CGHW22}.

First, we define a $k$-graph $H$ as follows.

\begin{cons}\label{Cons2}
	Let $t=\lfloor\frac{n-ks}{k-1}\rfloor-1$ and $T$ be a $t$-set. Then $H$ is the $k$-graph on $T\cup V(G)$ with $E(H)=E(G)\cup\{e\in\binom{V(G)\cup T}{k} : e\cap T\neq\emptyset\}$.
\end{cons}

\noindent It is easy to verify that the $k$-graph $H$ in Construction~\ref{Cons2} satisfies $\nu(H) = \nu(G) + t$ and $\delta_d(H) \ge \binom{n+t-d}{k-d} - \binom{n-d-s}{k-d} + 1$, where $t+s = \lfloor (n+t)/k \rfloor - 1$. Replacing $G$ with $H$, it suffices to consider the case $s = \lfloor n/k \rfloor - 1$. We then aim to find a matching covering all but at most $k-1$ vertices in $G$.

    To achieve this, we first consider the case when $G$ is close to the extremal graph and then deal with the case when $G$ is far from the extremal graph. Note that the graph $H(U,W)$ defined in Definition \ref{EX} is not the unique extremal graph. 

\begin{defi}\label{EX2}
    Given positive integers $1\leq p\leq k$, let $H^{p}(U,W)$ be the $k$-graph on vertex set $U\cup W$ with edge set $\{e\in\binom{U\cup W}{k}:1\leq|e\cap U|\leq p\}$.
\end{defi}

    Observe that $H(U,W)=H^{k}(U,W)$. For $|U|\leq|W|/(k-1)$, and $k-d\leq p\leq k$, we have $\nu(H(U,W))=\nu(H^{p}(U,W))=|U|$ and $\delta_{d}(H(U,W))=\delta_{d}(H^{p}(U,W))$. Thus, all these graphs are extremal graphs for the problem. If we further assume that the extremal graph has the minimum number of edges, then $H^{k-d}(U,W)$ will be the unique extremal graph. In what follows, we refer to $H^{k-d}(U,W)$ as the extremal graph.
    
	Returning to the proof of Theorem \ref{VDPC}. The case when $G$ is close to the extremal graph $H^{k-d}(U,W)$ for some partition $V(G)=U\cup W$ with $|U|=n/k$ has already been proved in \cite{LYY21} (See Lemma 2.3 in Section 2). For the case when $G$ is far from the extremal graph (i.e., $G$ is not $\zeta$-close to $H^{k-d}(U,W)$ for any partiton $V(G)=U\cup W$ with $|U|=n/k$, where $\zeta$ is a suitably chosen parameter), we will use the following absorbing lemma to find a powerful small matching $M_{abs}$ in $G$ and then use Lemma \ref{NPM} to find an almost perfect matching in the remaining graph.

\begin{lem}[Absorbing Lemma,\cite{HPS09}]\label{Abs}
	 For all $\gamma>0$ and integers $k>d>0$, there exists an $n_{0}$ such that for all integers $n>n_{0}$ the following holds: suppose that $H$ is a $k$-graph on $n$ vertices with $\delta_{d}(H)\geq(1/2+2\gamma)\binom{n-d}{k-d}$, then there exists a matching $M=M_{abs}$ in $H$, such that $|M|<\gamma^{k}n/k$, and that for every set $S\subseteq V(H)\setminus V(M)$ of size at most $\gamma^{2k}n$ and divisible by $k$, there exists a matching in $H$ covering exactly the vertices of $V(M)\cup S$.
\end{lem}	

	We apply Lemma \ref{NPM} here in a manner similar to its use in the previous section. We also need to show that, with high probability, the random subgraph of $H:=G\setminus V(M_{abs})$ contains a perfect fractional matching. Note that when $G$ is far from the extremal graph, choosing $\gamma$ sufficiently small ensures that $H$ is also far from the extremal graph. The following lemmas show that when $H$ is not close to the extremal graph, we have an upper bound of the independence number $\alpha(H)$ of $H$ and this property persists with high probability after randomization.

\begin{defi}
	Suppose that $k\geq3$ and $H$ is a $k$-graph. Let $\zeta>0$, and let $\mathcal{F}$ be a family of subsets of $V(H)$. We say that $H$ is $(\mathcal{F},\zeta)$-dense if $e(H[A])\geq \zeta e(H)$ for every $A\in \mathcal{F}$.
\end{defi}

\begin{lem}[Lemma 5.3 in \cite{LYY21}]\label{Ind}
	Let $k,d$ be integers with $k\geq 2$ and $d\in[k-1]$. Let $0<\zeta \ll 1$ and $\varepsilon\leq \zeta/8$. Let $n$ be a sufficiently large integer. Suppose $H$ is a $k$-graph with order $n$ such that $\delta_{d}(H)>\binom{n-d}{k-d}-\binom{n-d-n/k}{k-d}-\varepsilon n^{k-d}$, and $H$ is not $\zeta$-close to $H^{k-d}(U,W)$ for any partition of $V(H)$ in to $U,W$ with $|U|=n/k$. Then $H$ is $(\mathcal{F},\zeta/(2k!))$-dense, where $\mathcal{F}=\{A\subseteq V(H) : |A|\geq(1-1/k-\zeta/4)n\}$.
\end{lem}

\begin{lem}[Lemma 5.4 in \cite{LYY21}]\label{IndRan}
	 Let $c,\eta,\rho$ be positive real numbers and let $k,n$ be positive integers. Let $H$ be an $n$-vertex $k$-graph such that $e(H)\geq cn^{k}$ and $e(H[S])\geq \eta e(H)$ for all $S\subseteq V(H)$ with $|S|\geq \rho n$. Let $R\subseteq V(H)$ be obtained by taking each vertex of H independently and uniformly at random with probability $n^{-0.9}$. Then, for any positive real number $\theta \ll \rho$, the independence number of $H[R]$ is at most $(\rho+\theta+o(1))n^{0.1}$ with probability at least $1 - n^{O(1)}e^{-\Omega(n^{0.1})}$.
\end{lem}

	The next lemma provides other necessary properties of the random subgraph of $H$.

\begin{lem}[Lemma 5.5 in \cite{LYY21}]\label{random3}
	Let $k>d>0$ be integers with $k\geq 3$ and let $H$ be a $k$-graph on $n$ vertices. Let $R$ be chosen from $V(H)$ by taking each vertex uniformly at random with probability $n^{-0.9}$, and then arbitrarily deleting less than $k$ vertices so that $|R|\in k\mathbb{Z}$. Take $n^{1.1}$ independent copies of $R$ and denote them by $R^{i}$, $1\leq i\leq n^{1.1}$. For each $A\subseteq V(H)$ with $|A|\leq k$, let $Y_{A}=|\{i : A\subseteq R^{i}\}|$ and $\textrm{DEG}^{i}_{A}=|N_{H}(A)\cap \binom{R^{i}}{k-|A|}|$.
	Then, with probability $1-o(1)$, all of the following statements hold:
\begin{enumerate}[(1)]
	\item For every $v\in V(H)$,$Y_{\{ v\}}=(1+o(1))n^{0.2}$.
	\item $Y_{\{ u,v\}}\leq 2$ for every pair $\{u,v\}\subseteq V(H)$.
	\item $Y_{e}\leq 1$ for every edge $e\in E(H)$.
	\item For all $1\leq i\leq n^{1.1}$, we have $|R^{i}|=(1+o(1))n^{0.1}$.
	\item If $\varepsilon$ is a small constant with $0<\varepsilon\ll 1$, $\delta_{d}(H)\geq\binom{n-d}{k-d}-\binom{n-d-n/k}{k-d}-\varepsilon n^{k-d}$, then for all $i=1,\dots,n^{1.1}$, all $D\in\binom{V(H)}{d}$ and any positive real number $\xi \geq 2\varepsilon$, we have
	$$\textrm{DEG}^{i}_{D}>\binom{|R^{i}|-d}{k-d}-\binom{|R^{i}|-d-|R^{i}|/k}{k-d}-\xi |R^{i}|^{k-d}.$$
\end{enumerate}
\end{lem}	

	Subsequently, the following lemma shows that with the lower bound of $\delta_{d}(H[R])$ and the upper bound of $\alpha(H[R])$, we can prove that $H[R]$ contains a perfect fractional matching.

\begin{lem}\label{PFM2}
	Let $k,d$ be integers with $(k,d)\in\{(5,1),(6,2)\}$, and let $\zeta,\xi$ be positive reals with $\xi \ll \zeta$. Let $n$ be a positive integer such that $n$ is sufficiently large and $n\in k\mathbb{Z}$. Let $H$ be a $k$-graph of order $n$ such that $\delta_{d}(H)>\binom{n-d}{4}-\binom{n-d-n/k}{4}-\xi n^{4}$ and $\alpha(H)\leq (1-1/k-\zeta/5)n$. Then $H$ contains a perfect fractional matching.
\end{lem}

The proof of Lemma~\ref{PFM2} follows verbatim the proof of Lemma~4.3 in \cite{LYY21}, with Lemma~4.2 in \cite{LYY21} replaced by Theorem~\ref{StableK4}. We therefore omit the details.
Finally, we need the following lemma.

\begin{lem}\label{final}
	Let $n,k,d$ be integers with $(k,d)\in\{(5,1),(6,2)\}$. Let $\zeta\in(0,10^{-5})$ be a real number. Suppose that $G$ is an $n$-vertex $k$-graph with $\delta_{d}(G)\geq\binom{n-d}{k-d}-\binom{n-d-n/k+1}{k-d}+1$. If $G$ is not $2\zeta$-close to $H^{k-d}(U,W)$ for any partition $V(G)=U\cup W$ with $|U|=n/k$, then for sufficiently large $n$, $G$ contains a matching covering all but at most $k-1$ vertices in $G$.
\end{lem}

\begin{proof}[Proof of Lemma \ref{final}]
	First, note that for $(k,d)\in\{(5,1),(6,2)\}$, we always have $1-(1-1/k)^{k-d}>1/2$. So by Lemma \ref{Abs}, we get an absorbing matching $M_{abs}$ with $|V(M)|\leq \gamma^{k}n$, where $\gamma>0$ is chosen so that $\gamma^{k}(1-\gamma^{k})^{-4}<(\zeta/200)^{4}$. Let $H=G\setminus V(M_{abs})$. Then we have $\delta_{d}(H)\geq \binom{n'-d}{k-d}-\binom{n'-d-n'/k}{k-d}-\varepsilon(n')^{k-d}$, where $n'=|V(H)|$ and $\varepsilon=\gamma^{k}(1-\gamma^{k})^{-4}/3$. Furthermore, we also have that $H$ is not $\zeta$-close to $H^{k-d}(U',W')$ for any partition $V(H)=U'\cup W'$ with $|U'|=n'/k$.

	Now by Lemmas \ref{Ind}, \ref{IndRan} and \ref{random3}, we have random subsets $R^{i}$, $1\leq i\leq (n')^{1.1}$, which satisfy \textsl{(1)-(5)} in Lemma \ref{random3}. Moreover, for each $R^{i}$, $\alpha(H[R^{i}])\leq (1-1/k-\zeta/5)n'$. Then by Lemma \ref{PFM2}, $H[R^{i}]$ has a perfect fractional matching for all $1\leq i\leq (n')^{1.1}$.

	By Lemma \ref{random2}, we have a spanning subgraph $H''$ of $H$, such that $d_{H''}(v)=(1+o(1))(n')^{0.2}$ for each $v\in V(H)$, and $\Delta_{2}(H'')\leq (n')^{0.1}$. Finally, applying Lemma \ref{NPM}, we obtain a matching $M_{1}$, which covers all but at most $\gamma^{2k} n'$ vertices in $H$. Let $S$ denote the set of vertices in $H$ not covered by $M_{1}$. Then $|S|\leq \gamma^{2k} n'<\gamma^{2k} n$. Let $S'$ be a maximum subset of $S$ such that $|S'|\in k\mathbb{Z}$. Then $|S\setminus S'|\leq k-1$. By the property of $M_{abs}$, there exists a matching $M_{2}$ in $G$, which covers exactly the vertices in $V(M_{abs})\cup S'$. Therefore, $M_{1}\cup M_{2}$ is a matching in $G$, which covers all but at most $k-1$ vertices in $G$.
\end{proof}

 This lemma completes the proof of Theorem~\ref{VDPC}. \qedB

\begin{appendices}
\section{Appendix}\label{sec:Appendix}

\noindent In this section, we provide the proofs of the technical lemmas \ref{convex}, \ref{Prebound}, \ref{Maxvalue}, and \ref{Calculate}. In particular, Lemma \ref{Calculate} is verified using \textsl{Wolfram Mathematica}.

For a differentiable function $f(x)$ and an integer $i\geq 1$, we denote its $i^\text{th}$ derivative by $f^{(i)}(x)$. 

\subsection{Proof of Lemma \ref{convex}}    
\begin{proof}
	By the definition of $f(x)$, it suffices to prove that both $f_{1}(x)=\binom{m}{k-1}-\binom{m-\frac{(1-a)s}{1-x}}{k-1}$ and $f_{2}(x)=\binom{\frac{(k-1)(1-a)s}{1-x}+k-2}{k-1}$ are convex on $[0,(1+a)/2]$.

First consider $f_{1}(x)$. Let $g_{1}(x)=\ln\binom{m-\frac{(1-a)s}{1-x}}{k-1}$. Then we have
\begin{align*}
	g^{(1)}_{1}(x)&=-(1-a)s(1-x)^{-2}\left(\sum_{i=0}^{k-2}\frac{1}{m-i-\frac{1-a}{1-x}s}\right), \mbox{ ~~~ and }\\
	g^{(2)}_{1}(x)&=-2(1-a)s(1-x)^{-3}\left(\sum_{i=0}^{k-2}\frac{1}{m-i-\frac{1-a}{1-x}s}\right)-(1-a)^{2}s^{2}(1-x)^{-4}\left(\sum_{i=0}^{k-2}\frac{1}{(m-i-\frac{1-a}{1-x}s)^{2}}\right).
\end{align*}
	Note that $f_{1}(x)=\binom{m}{k-1}-e^{g_{1}(x)}$, we have $f^{(2)}_{1}(x)=-e^{g_{1}(x)}((g_{1}^{(1)}(x))^{2}+g^{(2)}_{1}(x))$.
	Therefore, we have
\begin{displaymath}
	f^{(2)}_{1}(x)=-\frac{(1-a)s}{(1-x)^{3}}\binom{m-\frac{1-a}{1-x}s}{k-1}\sum_{i=0}^{k-2}\left\{\left(m-i-\frac{1-a}{1-x}s\right)^{-1}
	\left(\sum_{0\leq j\leq k-2, j\neq i}\frac{m-j}{m-j-\frac{1-a}{1-x}s}-k\right)\right\}.
\end{displaymath}
	For $0\leq j\leq k-2$, Let 
\begin{displaymath}
	h_{j}(x)=\frac{m-j}{m-j-\frac{1-a}{1-x}s}-\frac{k}{k-2}.
\end{displaymath}
	Then we finally have
\begin{displaymath}
	f^{(2)}_{1}(x)=-\frac{(1-a)s}{(1-x)^{3}}\binom{m-\frac{1-a}{1-x}s}{k-1}\sum_{i=0}^{k-2}\left\{\left(m-i-\frac{1-a}{1-x}s\right)^{-1}
	\left(\sum_{0\leq j\leq k-2, j\neq i}h_{j}(x)\right)\right\}.
\end{displaymath}
	For $x\in[0,(1+a)/2]$, we have $\frac{1-a}{1-x}\leq 2$. Since $m\geq ks+k-2$, we have
\begin{displaymath}
	h_{j}(x)\leq\frac{m-(k-2)}{m-(k-2)-\frac{1-a}{1-x}s}-\frac{k}{k-2}\leq\frac{ks}{ks-2s}-\frac{k}{k-2}=0.
\end{displaymath}
	Therefore, $f^{(2)}_{1}(x)\geq 0$ for all $x\in[0,(1+a)/2]$. That is, $f_{1}(x)$ is convex on $[0,(1+a)/2]$.

	Now we consider $f_{2}(x)$. Let $g_{2}(x)=\ln f_{2}(x)$. Then we have
\begin{align*}
	g^{(1)}_{2}(x)&=(k-1)(1-a)s(1-x)^{-1}\sum_{i=0}^{k-2}\frac{1}{(k-1)(1-a)s+(k-2-i)(1-x)},\\
	g^{(2)}_{2}(x)&=(k-1)(1-a)s(1-x)^{-2}\sum_{i=0}^{k-2}\frac{(k-1)(1-a)s+2(k-2-i)(1-x)}{[(k-1)(1-a)s+(k-2-i)(1-x)]^{2}},
\end{align*}
	where $k-1\geq 2$ and $1-a>0$. So when $1-x>0$, we have $g^{(1)}_{2}(x)>0$ and $g^{(2)}_{2}(x)>0$.
	Since $f_{2}(x)=e^{g_{2}(x)}$, it is easy to see that
\begin{displaymath}
	f^{(2)}_{2}(x)=e^{g_{2}(x)}\left((g^{(1)}_{2}(x))^{2}+g^{(2)}_{2}(x)\right)>0,
\end{displaymath}
	and thus $f_{2}(x)$ is convex on $[0,1)$. 
    This completes the proof that $f(x)$ is convex on $[0,(1+a)/2]$.		
\end{proof}

\subsection{Proof of Lemma \ref{Prebound}}
\begin{proof}
	We first prove the case $\omega(1)<1/2$.
	Note that $\frac{1}{2}>\omega(1)\geq\omega(2)\geq\dots\geq\omega(m)$ and that $\omega(3s)=\omega(3s+1)=\dots=\omega(m)=0$. For all edges $e\in E(G)$, $|e\cap[3s-1]|\geq 3$.	Since $\nu(G[[3s-1]])\leq \nu(G)\leq \nu^{*}(G)\leq \sum_{i=1}^{m}\omega(i)\leq\mu s$, we have $e(G)=e(G[[3s-1]])\leq \max\{\binom{3s-1}{3}-\binom{3s-1-\mu s}{3},\binom{3\mu s+2}{3}\}$.

	For the case $\omega(1)\leq 3/5$. Similarly, for all edges $e\in E(G)$, $|e\cap[3s-1]|\geq 2$. Let $E_{1}=\{e\in E(G) \mid |e\cap[3s-1]|=2\}$. Then $e(G)=e(G[[3s-1]])+|E_{1}|$. Since we already have $e(G[[3s-1]])\leq \max\{\binom{3s-1}{3}-\binom{3s-1-\mu s}{3},\binom{3\mu s+2}{3}\}$, we only need to prove that $|E_{1}|\leq \binom{2\mu s}{2}\binom{m-3s+1}{1}$.

	Let $P=\{\{x,y\}\subseteq[3s-1] \mid \omega(x)+\omega(y)\geq 1\}$. It suffices to show that $|P|\leq \binom{2\mu s}{2}$.
    Let $\ell=\max_{\omega(i)\geq 2/5}\{i\}$. Then $\ell\leq 5\mu s/2$. Since $\omega(1)\leq 3/5$, $P\subseteq \binom{[\ell]}{2}$. Let $H$ be a graph on vertex set $[\ell]$ with edge set $P$. The restriction $\omega|_{[\ell]}$ is a fractional vertex cover of $H$, having size at most $\mu s$. Thus $\nu(H)\leq \mu s$. Then by Erd\H{o}s and Gallai's result \cite{EG59} (i.e., the EMC for $k=2$), we have $|P|\leq \binom{2\mu s}{2}$. 
	This proves the lemma, i.e., $e(G)\leq\max\{\binom{3s-1}{3}-\binom{3s-1-\mu s}{3},\binom{3\mu s+2}{3}\}+\binom{2\mu s}{2}\binom{m-3s+1}{1}$ for $\omega(1)\leq 3/5$.
\end{proof}

\subsection{Proof of Lemma \ref{Maxvalue}}
\begin{proof}
	It suffices to prove that for sufficiently large $m$, $g(s)$ is convex on $[0,(m-2)/(4-\delta)]$.
	Let
\begin{displaymath}
\begin{array}{l}
	g_{1}(s)=(1-a)s\left(\binom{m}{3}-\binom{m-\mu s}{3}\right)-\frac{(1-\delta)(1-a)s}{\beta}\binom{m-\mu s}{3}\\[0.3em]
	g_{2}(s)=(1-a)s\left(\binom{3s-1}{3}-\binom{3s-1-\mu s}{3}+(m-3s+1)\binom{2\mu s}{2}\right)-\frac{(1-\delta)(1-a)s}{\beta}\binom{m-\mu s}{3}\\[0.3em]
	g_{3}(s)=(1-a)s\left(\binom{3\mu s+2}{3}+(m-3s+1)\binom{2\mu s}{2}\right)-\frac{(1-\delta)(1-a)s}{\beta}\binom{m-\mu s}{3}\\[0.3em]
	g_{4}(s)=(1-a)s\left(\binom{3s-1}{3}-\binom{3s-1-\mu s}{3}\right)-\frac{(1-\delta)(1-a)s}{\beta}\binom{m-\mu s}{3}\\[0.3em]
	g_{5}(s)=(1-a)s\binom{3\mu s+2}{3}-\frac{(1-\delta)(1-a)s}{\beta}\binom{m-\mu s}{3}
\end{array}
\end{displaymath}
	Then it suffices to prove that, for sufficiently large $m$, each function $g_{i}(s)$ (for $1\leq i\leq 5$) is convex on the interval $[0,(m-2)/(4-\delta)]$. Let $\alpha=s/m$, then $0\leq \alpha \leq 1/(4-\delta)$. For each $1\leq i\leq 5$, let $C_{i}(\alpha)$ denote the coefficient of $m^{2}$ in $\frac{3\beta}{(1-a)\mu}g_{i}^{(2)}(s)$. Then it suffices to prove that $C_{i}(\alpha)>0$ for $0\leq\alpha\leq1/(4-\delta)$ and all $1\leq i\leq 5$. After computation, we have the following:
\begin{displaymath}
\begin{array}{ll}
	C_{1}(\alpha)= & (\beta+1-\delta)[6\mu^{2}\alpha^{2}-9\mu\alpha+3]\\[0.25em]
	C_{2}(\alpha)= & 6(27\beta-45\beta\mu+(\beta+1-\delta)\mu^{2})\alpha^{2}+9(4\beta-1+\delta)\mu\alpha+3(1-\delta)\\[0.25em]
	C_{3}(\alpha)= & 6(-36\beta\mu+(27\beta+1-\delta)\mu^{2})\alpha^{2}+9(4\beta-1+\delta)\mu\alpha+3(1-\delta)\\[0.25em]
	C_{4}(\alpha)= & 6(27\beta-9\beta\mu+(\beta+1-\delta)\mu^{2})\alpha^{2}-9(1-\delta)\mu\alpha+3(1-\delta)=\frac{1-\delta}{\beta+1-\delta}C_{1}(\alpha)+6\beta(28-9\mu)\alpha^{2}\\[0.25em]
	C_{5}(\alpha)= & 6(27\beta+1-\delta)\mu^{2}\alpha^{2}-9(1-\delta)\mu\alpha+3(1-\delta)=\frac{1-\delta}{\beta+1-\delta}C_{1}(\alpha)+162\beta\alpha^{2}
\end{array}
\end{displaymath}	

    Recall that $\delta=10^{-10}$, $0<5\varepsilon^{1/4}<1-a\leq \mu\leq 1$ and $\beta=(4-\delta)(1-b)-3(1-a)>0$.
    First, consider $C_{1}(\alpha)$. We have $C_{1}(\alpha)\geq C_{1}(1/(4-\delta))>C_{1}(1/3)=(\beta+1-\delta)(2\mu^{2}/3-3\mu+3)\geq2(\beta+1-\delta)/3>0$, as desired. 
    Next, by considering $C_{4}(\alpha)$ and $C_{5}(\alpha)$, we have $C_{4}(\alpha)>\frac{1-\delta}{\beta+1-\delta}C_{1}(\alpha)>0$ and $C_{5}(\alpha)>\frac{1-\delta}{\beta+1-\delta}C_{1}(\alpha)>0$.
	Finally, consider $C_{2}(\alpha)$ and $C_{3}(\alpha)$. In these cases, parameter $b\in[1/3,3/8)$. So we have $\beta>5(1-\delta)/8$, which implies $9(4\beta-1+\delta)\mu>0$. Thus, we have $C_{2}(\alpha)\geq\min\{C_{2}(0),C_{2}(1/(4-\delta))\}>0$ and $C_{3}(\alpha)\geq\min\{C_{3}(0),C_{3}(1/(4-\delta))\}>0$.
	Consequently, we have shown that $g(s)$ is convex on $[0,(m-2)/(4-\delta)]$.
\end{proof}

\subsection{Proof of Lemma \ref{Calculate}}
\begin{proof}
	Let $x=1-b$ and $y=1-a$. Then we have $5\varepsilon^{1/4}<y\leq x\leq 3/4$, $\mu=y/x\in(0,1]$ and $\beta=(4-\delta)x-3y\in(5(1-\delta)\varepsilon^{1/4},4)$. Set $z=\varepsilon^{1/4}<10^{-5}$.
	The original inequalities can then be rewritten in the following polynomial form. 
\begin{displaymath}
	(1-\delta)\left((4-\delta)x-y\right)^{3}-\left((4-\delta)x-3y\right)h(x,y)>\left((4-\delta)x-3y\right)p(x,z),
\end{displaymath}
where 
\begin{displaymath}
h(x,y)=\left\{\begin{array}{ll}
    (4-\delta)^{3}x^{3}-\left((4-\delta)x-y\right)^{3}, & 0< x\leq 5/8\\[0.25em]
    \max\{27yx^{2}-9y^{2}x+y^{3},27y^{3}\}+3(1-\delta)(4-\delta)y^{2}x, & 5/8<x\leq 2/3\\[0.25em]
    \max\{27yx^{2}-9y^{2}x+y^{3},27y^{3}\}, & 2/3<x\leq 3/4 \textrm{ },
\end{array}\right. 
\end{displaymath}   
and
\begin{displaymath}
p(x,z)=\left\{\begin{array}{ll}
    52(4-\delta)^{3}x^{3}z^{3}, &0< x\leq 5/8\\[0.3em]
    \frac{(4-\delta)^{3}x^{3}}{8100}, & 5/8<x\leq 3/4 \textrm{ }.
\end{array}\right.
\end{displaymath} 

    These polynomial inequalities of variable $x$, $y$ and $z$ have been verified through \textsl{WolframAlpha}(https://www.wolframalpha.com/). The corresponding program is available at \cite{PRO}.
\end{proof}

\end{appendices}


{
\small
\begin{thebibliography}{99}

\bibitem{AFHRRS12} N. Alon, P. Frankl, H. Huang, V. R\"{o}dl, A. Ruci\'{n}skid and B. Sudakov,  \newblock{Large matchings in uniform hypergraphs and the conjectures of Erd\H{o}s and Samuels}, \newblock{\emph{Journal of Combinatorial Theory, Series A}}, \textbf{119} (2012), 1200--1215.
\vspace{-0.6em}
\bibitem{BDE76} B. Bollob\'{a}s, D.E. Daykin and P. Erd\H{o}s, \newblock{Sets of independent edges of a hypergraph}, \newblock{\emph{The Quarterly Journal of Mathematics}}, \textbf{27(1)} (1976), 25--32.
\vspace{-0.6em}
\bibitem{CGHW22} Y. Chang, H. Ge, J. Han and G. Wang, \newblock{Matching of given sizes in hypergraphs}, \newblock{\emph{SIAM Journal on Discrete Mathematics}}, \textbf{36(3)} (2022), 2323--2338.
\vspace{-0.6em}
\bibitem{DH81} D.E. Daykin and R. H\"{a}ggkvist, \newblock{Degrees giving independent edges in a hypergraph}, \newblock{\emph{Bulletin of the Australian Mathematical Society}}, \textbf{23(1)} (1981), 103--109.
\vspace{-0.6em}
\bibitem{E65} P. Erd\H{o}s. \newblock{A problem on independent r-tuples}, \newblock{\emph{Ann. Univ. Sci. Budapest. E\"{o}tv\"{o}s Sect. Math}}, \textbf{8} (1965), 93--95.
\vspace{-0.6em}
\bibitem{EG59} P. Erd\H{o}s and T. Gallai, \newblock{On maximal paths and circuits of graphs}, \newblock{\emph{Acta Math. Acad. Sci. Hungar}}, \textbf{10} (1959), 337--356.
\vspace{-0.6em}
\bibitem{EKR61} P. Erd\H{o}s, C. Ko and R. Rado, \newblock{Intersection theorems for systems of finite sets}, \newblock{\emph{Quart. J. Math. Oxford}}, \textbf{12(2)} (1961), 313--320.
\vspace{-0.6em}
\bibitem{F87} P. Frankl, \newblock{The shifting technique in extremal set theory}, \newblock{\emph{Surveys in Combinatorics}}, \textbf{123} (1987), 81--110.
\vspace{-0.6em}
\bibitem{F13} P. Frankl, \newblock{Improved bounds for Erd\H{o}s' matching conjecture}, \newblock{\emph{Journal of Combinatorial Theory, Series A}}, \textbf{120(5)} (2013), 1068--1072.
\vspace{-0.6em}
\bibitem{F17A} P. Frankl, \newblock{Proof of the Erd\H{o}s matching conjecture in a new range}, \newblock{\emph{Israel Journal of Mathematics}}, \textbf{222(1)} (2017), 421--430.
\vspace{-0.6em}
\bibitem{F17B} P. Frankl, \newblock{On the maximum number of edges in a hypergraph with given matching number}, \newblock{\emph{Discrete Applied Mathematics}}, \textbf{216} (2017), 562--581.
\vspace{-0.6em}
\bibitem{FK22} P. Frankl and A. Kupavskii, \newblock{The Erd\H{o}s matching conjecture and concentration inequalities}, \newblock{\emph{Journal of Combinatorial Theory, Series B}}, \textbf{157} (2022), 366--400.
\vspace{-0.6em}
\bibitem{FR85} P. Frankl and V. R\"{o}dl, \newblock{Near perfect coverings in graphs and hypergraphs}, \newblock{\emph{European Journal of Combinatorics}}, \textbf{6(4)} (1985), 317--326.
\vspace{-0.6em}	
\bibitem{FRR12} P. Frankl, V. R\"{o}dl and A. Ruci\'{n}ski. \newblock{On the maximum number of edges in a triple system not containing a disjoint family of a given size}, \newblock{\emph{Combinatorics, Probability and Computing}}, \textbf{21} (2012), 141--148.
\vspace{-0.6em}
\bibitem{GLJ23} M. Guo, H. Lu and Y. Jiang, \newblock{Improved bound on vertex degree version of Erd\H{o}s matching conjecture}, \newblock{\emph{Journal of Graph Theory}}, \textbf{104(3)} (2023), 485--498.
\vspace{-0.6em}
\bibitem{H16A} J. Han, \newblock{Near perfect matchings in k-uniform hypergraphs II}, \newblock{\emph{SIAM Journal on Discrete Mathematics}}, \textbf{30(3)} (2016), 1453--1469.
\vspace{-0.6em}
\bibitem{H16B} J. Han, \newblock{Perfect matchings in hypergraphs and the Erd\"{o}s matching conjecture}, \newblock{\emph{SIAM Journal on Discrete Mathematics}}, \textbf{30(3)} (2016), 1351--1357.
\vspace{-0.6em}
\bibitem{HLS12} H. Hao, P.-S. Loh and B. Sudakov, \newblock{The size of a hypergraph and its matching number}, \newblock{\emph{Combinatorics, Probability and Computing}}, \textbf{21(3)} (2012), 442--450.
\vspace{-0.6em}
\bibitem{HPS09} H. H\`{a}n, Y. Person and M. Schacht, \newblock{On perfect matchings in uniform hypergraphs with large minimum vertex degree}, \newblock{\emph{SIAM Journal on Discrete Mathematics}}, \textbf{23(2)} (2009), 732--748.
\vspace{-0.6em}
\bibitem{HZ17} H. Huang and Y. Zhao, \newblock{Degree versions of the Erd\H{o}s-Ko-Rado theorem and Erd\H{o}s hypergraph matching conjecture}, \newblock{\emph{Journal of Combinatorial Theory, Series A}}, \textbf{150} (2017), 233--247.
\vspace{-0.6em}
\bibitem{K13}  I. Khan, \newblock{Perfect matchings in 3-uniform hypergraphs with large vertex degree}, \newblock{\emph{SIAM Journal on Discrete Mathematics}}, \textbf{27(2)} (2013), 1021--1039.
\vspace{-0.6em}
\bibitem{K16} I. Khan, \newblock{Perfect matchings in 4-uniform hypergraphs}, \newblock{\emph{Journal of Combinatorial Theory, Series B}}, \textbf{116} (2016), 333--366.
\vspace{-0.6em}
\bibitem{KO06} D. K\"{u}hn and D. Osthus, \newblock{Matchings in hypergraphs of large minimum degree}, \newblock{\emph{Journal of Graph Theory}}, \textbf{51(4)} (2006), 269--280.
\vspace{-0.6em}
\bibitem{KO09} D. K\"{u}hn and D. Osthus, \newblock{Embedding large subgraphs into dense graphs, in: S. Huczynka, J. Mitchell, C. Roney
Dougal (Eds.)}, \newblock{\emph{Surveys in Combinatorics, London Math. Soc. Lecture Note Ser}}, \textbf{365}, \newblock{Cambridge University
 Press}, 2009, 137--167.
\vspace{-0.6em}
\bibitem{KOT13} D. K\"{u}hn, D. Osthus and A. Treglown, \newblock{Matchings in 3-uniform hypergraphs}, \newblock{\emph{Journal of Combinatorial Theory, Series B}}, \textbf{103(2)} (2013), 291--305.
\vspace{-0.6em}
\bibitem{KOT14} D. K\"{u}hn, D. Osthus and T. Townsend, \newblock{Fractional and integer matchings in uniform hypergraphs}, \newblock{\emph{European Journal of Combinatorics}}, \textbf{38} (2014), 83--96.
\vspace{-0.6em}
\bibitem{LM14} T. \L uczak and K. Mieczkowska, \newblock{On Erd\H{o}s' extremal problem on matchings in hypergraphs}, \newblock{\emph{Journal of Combinatorial Theory, Series A}}, \textbf{124} (2014), 178--194.
\vspace{-0.6em}
\bibitem{LY22} H. Lu and X. Yu,  A note on exact minimum degree threshold for fractional perfect matchings, \emph{Graphs and Combinatorics}, \textbf{38(3)} (2022), 80.
\vspace{-0.6em}
\bibitem{LYY21} H. Lu, X. Yu and X. Yuan, \newblock{Nearly perfect matchings in uniform hypergraphs}, \newblock{\emph{SIAM Journal on Discrete Mathematics}}, \textbf{35(2)} (2021), 1022--1049.
\vspace{-0.6em}
\bibitem{MR11} K. Markstr\"{o}m and A. Ruci\'{n}ski, \newblock{Perfect matchings (and Hamilton cycles) in hypergraphs with large degrees}, \newblock{\emph{European Journal of Combinatorics}}, \textbf{32(5)} (2011), 677--687.
\vspace{-0.6em}
\bibitem{P08} O. Pikhurko, \newblock{Perfect Matchings and $K_{4}^{3}$-Tilings in Hypergraphs of Large Codegree}, \newblock{\emph{Graphs and Combinatorics}}, \textbf{24(4)} (2008), 391--404.
\vspace{-0.6em}
\bibitem{RRS06} V. R\"{o}dl, A. Ruci\'{n}ski and E. Szemer\'{e}di, \newblock{Perfect matchings in uniform hypergraphs with large minimum degree}, \newblock{\emph{European Journal of Combinatorics}}, \textbf{27(8)} (2006), 1333--1349.
\vspace{-0.6em}
\bibitem{RRS09} V. R\"{o}dl, A. Ruci\'{n}ski and E. Szemer\'{e}di, \newblock{Perfect matchings in large uniform hypergraphs with large minimum collective degree}, \newblock{\emph{Journal of Combinatorial Theory, Series A}}, \textbf{116(3)} (2009), 613--636.
\vspace{-0.6em}
\bibitem{TZ13} A. Treglown and Y. Zhao, \newblock{Exact minimum degree thresholds for perfect matchings in uniform hypergraphs II}, \newblock{\emph{Journal of Combinatorial Theory, Series A}}, \textbf{120} (2013), 1463--1482.
\vspace{-0.6em}
\bibitem{TZ16} A. Treglown and Y. Zhao, \newblock{A note on perfect matchings in uniform hypergraphs}, \newblock{\emph{The Electronic Journal of Combinatorics}}, \textbf{23(1)} (2016), 1--16.
\vspace{-0.6em}
\bibitem{PRO} \href{https://github.com/Yuze-Wu1997/Program-for-Published-Papers/blob/main/EMCK4.txt}{https://github.com/Yuze-Wu1997/Program-for-Published-Papers/blob/main/EMCK4.txt}

\end{thebibliography}
}
\end{document}